\documentclass{amsart}
\usepackage{amsmath,amsthm,amssymb,amsfonts,xypic}
 \usepackage[colorlinks=true]{hyperref}
\hypersetup{urlcolor=blue, citecolor=blue}

\def\arXiv#1{
   {\href{http://arxiv.org/pdf/#1}
   {{\mdseries\ttfamily arXiv:#1}}}}
\let\arxiv\arXiv

\newcommand{\al}{\alpha}    \newcommand{\be}{\beta}
    
\newcommand{\ep}{\epsilon}  
\newcommand{\la}{\lambda}   \newcommand{\La}{\Lambda}
    
\newcommand{\om}{\omega}    
\newcommand{\ga}{\gamma}    
\def\<{\langle}             \def\>{\rangle}
\newcommand{\R}{\mathbb{R}}\newcommand{\Z}{\mathbb{Z}}
\newcommand{\N}{\mathbb{N}}
\newcommand{\Sp}{{\mathbb S}}

\newcommand{\QED}{\hfill  {\vrule height6pt width6pt depth0pt}\medskip}

\newcommand{\supp}{\text{supp}}

\newcommand{\pt}{\partial_t}\newcommand{\pa}{\partial}
\newcommand{\D}{{\mathrm D}}\newcommand{\les}{{\lesssim}}

\newcommand{\beeq}{\begin{equation}}\newcommand{\eneq}{\end{equation}}

\theoremstyle{plain}
\newtheorem{thm}{Theorem}[section]

\newtheorem{coro}[thm]{Corollary}
\newtheorem{lem}[thm]{Lemma}

\theoremstyle{remark}
\newtheorem{rem}{Remark}[section]

\theoremstyle{definition}

\newenvironment{prf}{\noindent {\bf Proof.} }{\endprf\par}
\def \endprf{\hfill  {\vrule height6pt width6pt depth0pt}\medskip}

\newcommand{\lp}[2]{\Vert \, #1 \, \Vert_{#2}}

\numberwithin{equation}{section}

\title{Generalized and weighted Strichartz estimates}

\author{Jin-Cheng Jiang}
\address{Institute of Mathematics\\Academic Sinica\\
 Taipei, Taiwan 10617, R.O.C}
\curraddr{Department of Mathematics\\National Tsing Hua University\\Hsinchu, Taiwan 30013, R.O.C}
 \email{jcjiang@math.nthu.edu.tw}

\author{Chengbo Wang}
\address{Department of Mathematics\\
            Johns Hopkins University\\
            Baltimore, MD 21218, USA}
\curraddr{Department of Mathematics\\
                Zhejiang University\\
                Hangzhou 310027, China}
\email{wangcbo@gmail.com}
\urladdr{http://www.math.zju.edu.cn/wang/}
\thanks{The second author was supported by the Fundamental Research Funds for the Central Universities, NSFC 10871175 and 10911120383.}

\author{Xin Yu}
\address{Department of Mathematics\\Johns Hopkins University\\
Baltimore, MD 21218, USA}
\email{xyu@math.jhu.edu}

\dedicatory{}
\date{}

\begin{document}
\maketitle

\begin{abstract}
In this paper, we explore the relations between different kinds of Strichartz estimates
and give new estimates in Euclidean space $\mathbb{R}^n$. In particular, we prove the generalized and weighted Strichartz estimates for a large class of dispersive operators including the Schr\"{o}dinger and wave equation. As a sample application of these new estimates,
we are able to prove the Strauss conjecture with low regularity for dimension $2$ and $3$.
\end{abstract}

\tableofcontents

\section{Introduction}

In this paper, we explore the relations between different kinds of Strichartz estimates
and give new estimates in Euclidean space $\mathbb{R}^n$. In particular, we prove the generalized and weighted Strichartz estimates for a large class of dispersive operators including the Schr\"{o}dinger and wave equation. As a sample application of these new estimates, we are able to prove the Strauss conjecture with low regularity for dimension $2$ and $3$. In some sense, this paper can be viewed as a sequel to the work of Fang and the second author \cite{FaWa}.

Let $\D=\sqrt{-\Delta}$. Typically, Strichartz estimates for the dispersive operators
$e^{itD^a},a=1,2$, are a family of estimates which state
\begin{equation}
\|e^{itD^a}f\|_{L^q_tL^r_x(\mathbb{R}\times\mathbb{R}^n)}\leq C\|f\|_{\dot{H}^s},
\end{equation}
where $\dot{H}^s$ ($s<n/2$) denotes the homogenous Sobolev space in $\mathbb{R}^n$.
These estimates were first established by Strichartz~\cite{Str77} for $q=r$. They were generalized to non-endpoint admissible $(q,r)$ by Ginibre and Velo~\cite{GiV84}~\cite{GiV95}, Lindblad and Sogge~\cite{LdSo95}.
The end point estimates were proven by Keel and Tao~\cite{KeTa98}.
They are powerful tools in the study of the nonlinear Schr\"{o}dinger and wave equations. See for example Cazenave~\cite{Ca03}, Sogge~\cite{So08} and Tao~\cite{Tao06} and references therein.

Let $\Delta_{\om}=\sum_{1\le i<j \le n}
\Omega_{ij}^2$ be the Laplace-Beltrami operator on the unit sphere $\Sp^{n-1}\subset
\R^n$ with $\Omega_{ij}=x_i
\partial_j-x_j \pa_i$, $\om\in \Sp^{n-1}$, and define $\La_\om=\sqrt{1-\Delta_{\om}}$. Based on the usual Sobolev
spaces $H^s$, we introduce the Sobolev spaces with angular
regularity as follows ($b\ge 0$)
$$H^{s,b}_{\om}=\La_\om^{-b}H^s=\{u\in H^s :
\| \La^b_\om u \|_{H^s}<\infty\}.$$ For the homogeneous Sobolev space
$\dot{H}^s=\D^{-s}L^2$, we can similarly define the space
$\dot{H}^{s,b}_{\om}=\La_\omega^{-b}\D^{-s} L^2$. We will also use the homogeneous Besov space $\dot B^s_{p,q}$ (for $sp<n$ or $sp=n$ and $q=1$), which is defined to be the completion of $C_0^\infty$ in $\mathcal{S}'$, with respect to the norm
$\|f\|_{\dot B^s_{p,q}}=\| 2^{sk} S_k f\|_{\ell_k^q L^p}$. Here $S_k$
are the Fourier multiplier operators of the homogenous Littlewood-Paley decomposition.

Recently, there have been many works on various generalizations of the Strichartz estimates and their applications. Before stating our results and related works, we would like to list different types of Strichartz estimates by the following table. After tagging different Strichartz estimates, it will be easier to explain the history and give an overview of our work by diagrams
and lists afterward. Here, in general, we will be able to consider the operators $e^{itD^a}$ with $a>0$. Recall that the operators $e^{itD^a}$ are related to the Schr\"{o}dinger equation ($a=2$) and the wave equation ($a=1$). Also we denote $L^q_TL^r_x$ as $L^q_t L^r_x$ with domain $[0,T]\times\R^n$.

\begin{center}
  Table 1. Different types  of Strichartz estimates.
\end{center}
\noindent\begin{tabular}{|lllll|}
\hline
Name & Left Norm & Right Norm & Range of $q,r$ ($q,r\ge2$) & No. \\
\hline
 Strichartz & $ L^q_tL^r_x $ & $\dot{H}^s (H^s)$  & $\frac 1q\leq \frac{n-1}2(\frac 12-\frac 1r)$ &  (I)\\
G. Strichartz    & $ L^q_tL^r_x $ & $\dot{H}^{s,1/q}_{\omega} $ &
$\frac 1q< (n-1)(\frac 12-\frac 1r)$ & (II) \\
G. Strichartz &$L^q_t L^r_{|x|}L^2_{\omega}$ & $\dot{H}^s $ &
$\frac 1q< (n-1)(\frac 12-\frac 1r)$ & (III) \\
Loc. G. Strichartz  & $ L^q_TL^r_x $ & $\dot{H}^{s,b}_{\omega} $ &
$(n-1)(\frac 12-\frac 1r)\leq \frac 1q$  & (IV) \\
Loc. G. Strichartz & $L^q_TL^r_{|x|}L^2_{\omega}$ & $\dot{H}^s $ &
$(n-1)(\frac 12-\frac 1r)\leq \frac1q$ & (V) \\
Morawetz-KSS & $\langle x\rangle^{\alpha}L^2_{T,x}$ & $L^2_x$ & $\alpha\geq 0$ & (VI)\\
W. Strichartz & $ |x|^{\alpha}L^q_tL^r_{|x|} L^2_\omega $ & $\dot{H}^{s,b}_{\omega} $ &
$\frac{1}{q}-\frac{n-1}{2}+\frac{n-1}{r}< \alpha < \frac{n}{r} $ & (VII) \\
W. Strichartz & $ |x|^{\alpha}L^q_tL^r_x $ & $\dot{H}^{s,b}_{\omega} $ &
$\frac{1}{q}-\frac{n-1}{2}+\frac{n-1}{r}< \alpha < \frac{n}{r} $ & (VIII) \\
 \hline
\end{tabular}

Now let us give a brief history of the generalized Strichartz
estimates (G. Strichartz) and weighted Strichartz estimates (W.
Strichartz) within our best knowledge. The generalized Strichartz
estimates were first studied in the endpoint case of the classical
Strichartz estimates for the wave and Schr\"{o}dinger equations.
For the wave equation ($a=1$), it is known that the 3-dimensional
endpoint $L^2_t L^\infty_{x}$ Strichartz estimate fails
(\cite{KlMa93}), however, the corresponding generalized estimates
(II) and (III) were proven in \cite{MaNaNaOz05}. For the
Schr\"{o}dinger equation ($a=2$),
Montgomery-Smith \cite{Montgomery98} proved the failure of the $L^2_t L^\infty_x$ Strichartz
estimate and Tao \cite{Tao00} proved the
2-dimensional endpoint $L^2_t L^\infty_{|x|}L^2_\omega$ estimate.

Then, for the wave equation, generalized Strichartz estimates of type (II) were proven in Sterbenz \cite{St05} ($n\ge 3$) and Fang and Wang \cite{FaWa} ($n\ge 2$). When the initial data is radial, the localized estimates (IV) have also been obtained in Hidano-Kurokawa \cite{HiKu-p}, by proving certain weighted radial Hardy-Littlewood-Sobolev estimates. For the 2-dimensional wave equation, the generalized Strichartz estimates of type (III) and (V) were proven recently by Smith, Sogge and Wang \cite{SmSoWa09} and Fang and Wang \cite{FaWa5} respectively.

Around the same time, Keel, Smith and Sogge \cite{KeSmSo02} proved the estimates (VI) for  the wave equation when $\al=1/2$, $n=3$, which were named KSS estimates or Morawetz-KSS estimates. The estimates of this type were developed drastically afterwards (see e.g. \cite{Met04-1}, \cite{HiYo05}, \cite{MetSo06}, \cite{SW} and \cite{HiWaYo}).

In some sense, the Morawetz-KSS estimates can be viewed as a special case of the weighted Strichartz estimates (VII) and (VIII). The weighted Strichartz estimates (VII) with $q=r$ were proven by Fang and Wang \cite{FaWa} ($a>0$) and Hidano, Metcalfe, Smith, Sogge and Zhou \cite{HMSSZ} ($a=1$), with the previous work for $a=1$ and radial data in \cite{Hi07}.

As was clear from \cite{FaWa} and \cite{HMSSZ}, these estimates are intimately related with each other.
A starting point can be the homogenous trace lemma (H.T.L., see (1.3) of \cite{FaWa}), i.e.,
\begin{equation}\label{htl}
r^{(n-b)/2}\|f(r\omega)\|_{L^2_{\omega}}\leq C\|D^{b/2}\Lambda^{(1-b)/2}_{\omega}f\|_{L^2_x}\ ,\ b\in (1,n)\ . \end{equation}
Using this and interpolation, we can conclude the
case $q=r$ in (VII) as follows,

$$\begin{array}{c}\xymatrix{
{\rm H.T.L.} \ar[r] \ar[dr]& |x|^{\alpha}L^2_{t,x}\ (\frac 12<\al<\frac n2) \ar[r] &  q=r\;\; {\rm in\;\;(II), (III), (VII)} \\
   & |x|^{\alpha}L^{\infty}_{t}L^{\infty}_{|x|}L^2_{\omega}\ (-\frac n2<\al<0) \ar[ur] }\\
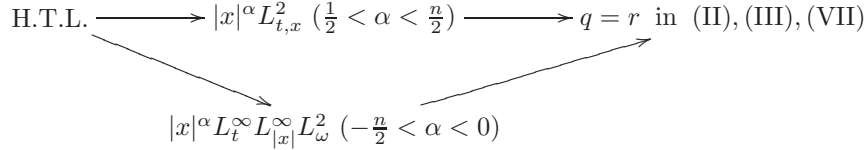
   \textrm{ Figure 1. Weighted Strichartz estimates from H.T.L..}\end{array}$$

Similarly, for the wave equation $a=1$, by using the inhomogenous trace lemma (I.T.L., see (1.7) in \cite{FaWa}), i.e.,
\begin{equation}\label{itl}
r^{(n-1)/2}\|f(r \om)\|_{L^2_{\om}}\leq C\|f\|_{\dot{B}^{1/2}_{2,1}}\ ,
\end{equation}
we can conclude a couple of estimates
in (VI), (II) and (VIII) (see \cite{St05} for the Rodnianski's argument to deduce (II) from (VI.a)).

$$\begin{array}{c}\xymatrix{
{\rm I.T.L} \ar[r]  \ar[d]& H^{1/2+}\subset\dot{B}^{1/2}_{2,1}\subset L^2_{\omega}&
\\
{\rm Loc. \;\;Energy} \ar[r]\ar[dr]^{+{\rm Energy}} \ar[d]^{+{\rm Energy}}&
 \langle x\rangle^{1/2+} L^2_{t,x} {\rm\ in\ (VI.a)} \ar[r]&(II)\\
|x|^{\alpha}L^2_{T}L^2_x  {\rm\ in\ (VI.c)} (0
 \le \al<1/2) \ar[d]^{+{\rm I.T.L}}&
\langle x\rangle^{1/2}L^2_{T}L^2_x  {\rm\ in\ (VI.b)}& \\
 |x|^{\alpha}L^p_{T}L^p_x  {\rm \ in\ (VIII.Loc)}&&
 }\\
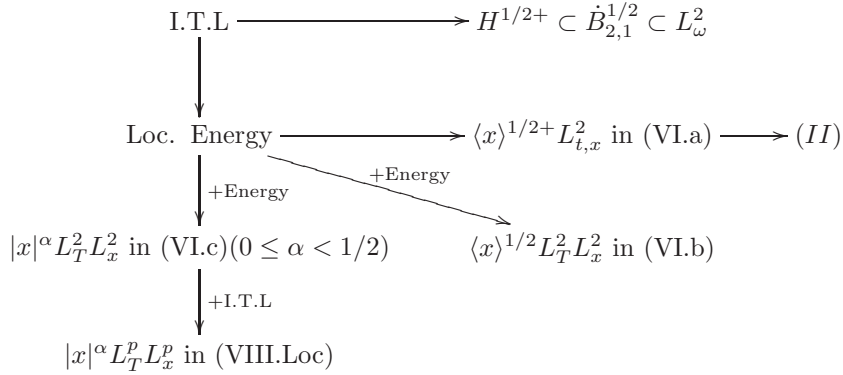
 \textrm{Figure 2. Some consequences of I.T.L..}
  \end{array}$$

As long as the estimates in the above diagrams were built, we can get our results as follows.
\begin{itemize}
\item  Use (VII) with $q=r$ (\cite{FaWa}) and Rodnianski's argument, we can get estimates in (II) for $n\geq 2$ and $a>0$, i.e.
       Theorem~\ref{9-thm-Stri-Rod}.
\item  Use (VI.b), (VI.c) and Rodnianski's argument, to get estimates in (IV) with $a=1$, i.e. Theorem~\ref{9-thm-Stri-Rod-local}.
\item  Use the arguments of \cite{SmSoWa09} and \cite{FaWa5}, to get estimates in (III) and (V) with $a=1$, i.e. Theorem~\ref{8-thm-StriSt}.
\item  Use the argument of Sterbenz \cite{St05}, to get estimates in (VIII) for $a=1$ and radial functions, i.e. Theorem~\ref{full_weig_str_th_sph_data}.
\item  Use $q=r$ in (VII), Rodnianski's argument, together with a localized version of
the weighted Hardy-Littlewood-Sobolev inequality, to get estimates in the range of $q\le r$ in (VIII), i.e.
    Theorem~\ref{8-thm-StriRod-Weig1}.
\item  Use $q=r$ in (VI) and Hardy's inequality, to get estimates in the range of $q\ge r$ in (VII), i.e. Theorem~\ref{8-thm-StriRod-Weig2}.
\end{itemize}

We now state our results precisely.
First we would like to give an angular generalization of the classical Strichartz estimates.
\begin{thm}[Generalized Strichartz
Estimates]\label{9-thm-Stri-Rod} Let $n\ge 2$,  $a>0$, $q,r\ge 2$ and $r<\infty$. If
$$\frac{1}{q}<(n-1)\left(\frac{1}{2}-\frac{1}{r}\right)\textrm{ or } (q,r)=(\infty,2),$$
then we have \beeq\label{9-est-Stri-Rod}\|e^{i t \D^a} f\|_{L^q_t
L^r_x} \les \|f\|_{ \dot{H}^{s,{\frac 1 q}}_\omega}\eneq with
$s=n(\frac{1}{2}-\frac{1}{r})-\frac{a}{q}$.
\end{thm}

\begin{rem}
  The technical restriction $r<\infty$ can essentially be removed (except the endpoint $(q,r)=(\infty,\infty), (2,\infty)$), if we use the real interpolation argument as in Section 5.2.
\end{rem}

\begin{rem}
  In the case of the wave equation ($a=1$), recall
that we have the classical Strichartz estimates (see e.g.
\cite{FaWa2}, \cite{KeTa98})
$$\|e^{i t \D} f\|_{
L^q_t L^r_{x} } \les \| f\|_{\dot{H}^s},$$ under the admissible condition
 \beeq\label{9-est-Stri-wave-admi}\frac{1}{q}\le
\min\left(\frac{1}{2},\frac{n-1}{2}\left(\frac{1}{2}-\frac{1}{r}\right)\right),
(q,r)\neq (\max(2,\frac{4}{n-1}),\infty), (q,r)\neq (\infty,\infty).
\eneq
The result in Theorem \ref{9-thm-Stri-Rod} extends the Strichartz
estimates to the case of
$$\frac{1}{q}<(n-1)(\frac{1}{2}-\frac{1}{r})$$ by requiring some
additional angular regularity on the data.\end{rem}

\begin{rem}
The requirement for $(q,r)$ are sharp for $a=1$, see Remark \ref{rq-1.5o}. When $a\neq 1$, the sharpness may be different. It will be interesting to determine the sharp range for $(q,r)$, at least for the Schr\"{o}dinger equation ($a=2$). Recall that for the Schr\"{o}dinger equation ($a=2$), the Strichartz estimates can be stated as follows (see e.g. \cite{KeTa98}, (1.26) of \cite{FaWa})
 \beeq\label{60-est-Stri-Schro}\|e^{i t \Delta}f\|_{L^q_t L^r_x}\les
\| f\|_{\dot H^s},\eneq for
\beeq\label{60-est-Schro}\frac{1}{q}\le
\min(\frac{1}{2},\frac{n}{2}\left(\frac{1}{2}-\frac{1}{r}\right)), (q,r)\neq
(\infty,\infty), (2,\infty).\eneq
We note here that for $n=a=2$, our estimate is worse than the standard
one.
\end{rem}

In fact, for the wave equation ($a=1$), we can improve the required angular regularity to be
almost optimal for the non-admissible $(q,r)$, by interpolating with
the classical Strichartz estimates, which recover the results in
\cite{St05} for $n\ge 3$ and \cite{FaWa} for the full range $n\ge
2$.
\begin{coro}\label{9-thm-Stri-Ster}
Let $n\ge 2$, $$s=n(\frac{1}{2}-\frac{1}{r})-\frac{1}{q},\
s_{kn}=\frac{2}{q}-(n-1)(\frac{1}{2}-\frac{1}{r}),$$ and
\beeq\label{60-est-Conj1-wave}\frac{n-1}{2}(\frac{1}{2}-\frac{1}{r})<\frac{1}{q}<(n-1)(\frac{1}{2}-\frac{1}{r}),\
q\ge 2.\eneq Then we have the estimates \beeq \label{sec6.2}\|e^{i t
\D} f\|_{L^q_t L^r_{x}}\les \| f\|_{\dot{H}^{s,b}_{\om}}\eneq
 for any $b>s_{kn}$.
\end{coro}

On the other hand, if we localize the domain in finite time interval $[0,T]$, by making use of KSS estimates, we will get the following localized Strichartz estimates for the wave equation.

\begin{thm}\label{9-thm-Stri-Rod-local} Let $n\ge 2$, $2\le r<\infty$ and
$$ \frac{1}{q}=(n-1)(\frac{1}{2}-\frac{1}{r}) \le \frac 12\ .$$
Then we have \beeq\label{8-est-StriRod-bdry}\|e^{i t \D} f\|_{L^q_T
L^r_x} \les (\ln(2+T))^{1/q}
\|(\ln(2+\D))^{1/q}\D^{\frac{1}{2}-\frac{1}{r}}\La_\omega^{1/q}
f\|_{ L_x^2}.\eneq For the endpoint case $(q,r,n)=(2,\infty,2)$ and
any $\ep>0$, we have \beeq\label{8-est-StriRod-bdryenpt}\|e^{i t \D}
f\|_{L^2_T L^\infty_x} \les (\ln(2+T))^{\frac{1}{2}} \| f\|_{
H^{\frac{1}{2}+\ep,\frac 12+\ep}_\omega}.\eneq Moreover, if
$(n-1)(\frac{1}{2}-\frac{1}{r})<\frac{1}{q} \le \frac 1 2$, we have
\beeq\label{8-est-StriRod-local}\|e^{i t \D} f\|_{L^q_T L^r_x} \les
T^{\frac{1}{q}-(n-1)(\frac{1}{2}-\frac{1}{r})}
\|
f\|_{ \dot H_\omega^{s,b}}\ ,\eneq with $s=\frac{1}{2}-\frac{1}{r}$ and $b=(n-1)(\frac{1}{2}-\frac{1}{r})$.
\end{thm}

In the recent work of Smith, Sogge and Wang \cite{SmSoWa09} (see
also Fang and Wang \cite{FaWa5}), we see that when $n=2$, we can
in fact improve further the generalized Strichartz estimates for
the wave equation to the $L^q_t L^r_{|x|} L^2_\omega$ estimates, in
which case, the angular regularity is not required.
Here $L^r_{|x|}$ denotes the Lebesgue space for the variable $|x|$ with respect to the
measure $|x|^{n-1}d |x|$.
Inspired by their work, we generalize the results to the general spatial dimensions.

\begin{thm}\label{8-thm-StriSt}
  Let $n\ge 2$ and $q,r \ge2$.
  If
$$\tfrac1q <(n-1)(\tfrac12-\tfrac1r) \textrm{ or }(q,r)=(\infty,2),
(q,r)\neq (2,\infty), (q,r)\neq (\infty,\infty),$$
 then we have \beeq\label{8-est-StriSt-2D}\|e^{i t\D}
f\|_{L^q_t L^r_{|x|}L^2_\omega} \les \|f\|_{\dot{H}^s}\eneq with
$s=\frac n 2-\frac{n}{r}-\frac{1}{q}$. On the
other hand, if $s=\frac 12-\frac 1r$ and $\frac{1}{q} >
(n-1)(\frac{1}{2}-\frac{1}{r})$, then
\beeq\label{8-est-StriSt-2D2}\|e^{i t\D} f\|_{L^q_T
L^r_{|x|}L^2_\omega} \les T^{\frac 1q-(n-1)(\frac 12-\frac
1r)}\|f\|_{\dot{H}^s}.\eneq
  Moreover, if $\frac{1}{q}=
(n-1)(\frac{1}{2}-\frac{1}{r})$, then for any $\delta>0$, we have
\beeq\label{8-est-StriSt-bdry}\|e^{i t \D} f\|_{L^q_T
L^r_{|x|}L^2_\omega} \les (\ln(2+T))^{1/q} \| f\|_{
H_x^{\frac{1}{2}-\frac{1}{r}+\delta}}.\eneq
\end{thm}
\begin{rem}
As a complement, we cite the endpoint Strichartz estimates when $n=2$ here. For $2<q<\infty$, Smith, Sogge and Wang \cite{SmSoWa09} prove that
\begin{equation}\label{9-est-Stri-SSW}
\bigl\|\, e^{itD}f \,  \bigr\|_{L^q_t L^\infty_{|x|}
L^2_\omega({\mathbb R}\times {\mathbb R}^2)} \le C_{q}\|f\|_{\dot
H^{\gamma}({\mathbb R}^2)}, \quad \gamma= 1-1/q.
\end{equation}
Moreover, Fang and Wang \cite{FaWa5} prove the
endpoint estimates for $q=2$,
\begin{equation}\label{9-est-Stri-FW5}
\bigl\|\, e^{itD}f \,  \bigr\|_{L^2_t L^\infty_{|x|}
L^2_\om([0,T]\times {\mathbb R}^2)} \le C_\gamma
\left(\ln(2+T)\right)^\frac12\|f\|_{ H^{\gamma}({\mathbb R}^2)}
\end{equation} for any $\gamma>1/2$.
\end{rem}
\begin{rem}
For the frequency localized functions, the estimate \eqref{8-est-StriSt-2D} holds for any $q,r \ge2$ such that $\tfrac1q <(n-1)(\tfrac12-\tfrac1r)$ (see \eqref{GS-a}).
\end{rem}
\begin{rem}
The same type of estimate to \eqref{8-est-StriSt-2D} was
proved in \cite{MaNaNaOz05} for $(q,r,n)=(2,\infty,3)$. We also remark that a similar estimate was proved for the Klein-Gordon equation and Schr\"{o}dinger equation with $(q,r,n)=(2,\infty,2)$ in \cite{KaOz11}.
\end{rem}

Next, we are interested in exploiting the weighted Strichartz
estimates for any $q,r\ge 2$. The fist result is for the wave equation and radial initial data, which will serve as a guideline for the general estimates.

\begin{thm}[Weighted Strichartz estimates for radial initial data]
\label{full_weig_str_th_sph_data} Let $q,r\ge 2$, and $u$ be a
radial function on $\mathbb{R}^{n+1}$ such that $\Box
u =(\pt^2-\Delta)u= 0$. Then the following estimates hold with $s=\al+n(\frac
12-\frac 1r)-\frac 1q$:
\begin{equation}
    \lp{|x|^{-\al} u}{L^q_t L_x^r} \ \lesssim \ \lp{u(0)}{\dot H^s} +
    \lp{\partial_t u\, (0)}{\dot H^{s-1}} \ ,
    \label{full_sph_str_est}
\end{equation}
when
$$\left\{\begin{array}{ll}
\frac{1}{q}-(n-1)(\frac{1}{2}-\frac{1}{r})<\al<\frac{n}{r},&2\le
q,r<\infty,\\ -(n-1)(\frac{1}{2}-\frac{1}{r})\le \al<\frac{n}{r}&q=\infty, 2\le r<\infty,\\
\frac{1}{q}-\frac{n-1}{2}< \al\le 0&r=\infty, 2< q<\infty.
\end{array}\right.$$
\end{thm}
\begin{rem}\label{rq-1.5o}
The requirement on $\al$ is essentially optimal. In fact, since the decay estimates for the wave equation are sharp in general even for radial functions (see e.g. Lemma 4.1 of \cite{HiKu-p}). By those estimates, it is easy to see that to bound the left hand side of \eqref{full_sph_str_est}, we must have $$\frac 1
q-\frac{n-1}{2}+\frac{n-1}{r}<\al<\frac{n}{r},\ q,r<\infty,$$
$$\frac 1
q-\frac{n-1}{2}<\al\le 0,\ q<r=\infty,$$
$$-\frac{n-1}{2}+\frac{n-1}{r}\le \al< \frac{n}{r},\ r<q=\infty.$$
\end{rem}

The second result is the weighted Strichartz estimates for general $a>0$ and general data.
\begin{thm}\label{8-thm-StriRod-Weig1}
Let \beeq\label{qrcondition}2\le  q\le r<\infty,~\text{and}~\frac 1
q-\frac{n-1}{2}+\frac{n-1}{r}<\al<\frac{n}{r}.\eneq Then we have the
following weighted Strichartz estimates,
 \beeq\label{60-est-Stri-weighted21}\| |x|^{-\al}
e^{i t \D^a} f\|_{L^q_t L^r_x} \les \| f\|_{\dot H^{s,b}_\omega
},\eneq where $$-\al+\frac a q+\frac n r=-s+\frac n 2,~b={\frac 1 q}-\al.$$
\end{thm}

In Theorem \ref{8-thm-StriRod-Weig1}, we have an additional restriction $q\le r$
for $q,r$, compared with \eqref{full_sph_str_est}. In general, we can relax this restriction.
\begin{thm}\label{8-thm-StriRod-Weig2}
Let $2\le r\le q\le \infty$, and $\frac 1
q-\frac{n-1}{2}+\frac{n-1}{r}<\al<\frac{n}{r}$. Then we have the
following weighted Strichartz estimates,
 \beeq\label{60-est-Stri-weighted22}\| |x|^{-\al}
e^{i t \D^a} f\|_{L^q_t L^r_{|x|}L^2_\omega} \les \| f\|_{ \dot H_\omega^{s,b}},\eneq where $$-\al+\frac a q+\frac n
r=-s+\frac n 2$$ and $$\left\{\begin{array}{ll} b \ge -\al+\frac 1q-(n-1)\left(\frac 12-\frac 1r
\right),& \textrm{if } \frac 1 q-\frac{n-1}2+\frac{n-1}r<\al<\frac n q,\\
b>-\frac
{n-1}q-(n-1)\left(\frac 12-\frac
1r\right), &\textrm{if } \frac n q \le \al<\frac n
r.  \end{array}\right .$$
\end{thm}
\begin{rem}
The results in Theorem \ref{8-thm-StriRod-Weig1} and Theorem \ref{8-thm-StriRod-Weig2} are generalizations of the weighted Strichartz estimates in Fang and Wang \cite{FaWa}(see also \cite{HMSSZ} for $a=1$), i.e.
\beeq\label{9-est-Stri-FW6}\||x|^{\frac{n}{2}-\frac{n}{r}-\frac{b}{2}}
e^{i t \D^a} f (x) \|_{L^r_{t, |x|} L^2_\omega } \les
\|\D^{\frac{b}{2}-\frac{a}{r}}\La_\omega^{\frac{1-b}{2}}f\|_{L^2_x},
\eneq if $b\in (1,n) $ and $r\in [2,\infty]$. In particular, when $r=2$, we have the
generalized Morawetz estimates
\beeq\label{9-est-Morawetz-homo}\||x|^{-\frac{b}{2}} e^{i t \D^a}
f\|_{L^2_{t,x} }\les
\|\D^{\frac{b-a}{2}}\La_\omega^{\frac{1-b}{2}}f\|_{L^2_x},\eneq for
any $b\in(1,n)$ and $a>0$.
\end{rem}

Lastly, we present local in time weighted Strichartz estimates for the wave equation.
\begin{thm}
\label{8-thm-StriRod-Weig3} Let $n\geq 2$, $2\le p<\infty$, and
$(n-1)(\frac{1}{p}-\frac {1}2)<\al<\frac n p-\frac{n-1}2$, then we
have the following weighted Strichartz estimates,
 \beeq\label{60-est-Stri-weighted23}\| |x|^{-\al}
e^{i t \D} f\|_{L^p_t L^p_{|x|}L^2_\omega} \les
T^{-\al-\frac{n-1}2+\frac n p}\|
f\|_{ \dot H^{s}},\eneq where $s=\frac 12-\frac 1p$.
\end{thm}
This theorem comes from an interpolation between KSS estimates and the
inhomogenous trace lemma \eqref{itl}, as in \cite{Yu09}.

\begin{rem}
We can also get more general estimates, if we interpolate the KSS estimates with the weighted Strichartz estimates \eqref{9-est-Stri-FW6} with $r=\infty$.
\end{rem}

This paper is arranged as follows. In section 2 we prove the estimates stated in Theorem \ref{9-thm-Stri-Rod} and Theorem \ref{9-thm-Stri-Rod-local}; In section 3 we prove Theorem \ref{8-thm-StriSt}; In section 4 we prove Theorem \ref{full_weig_str_th_sph_data}; In section 5 we prove Theorem \ref{8-thm-StriRod-Weig1} and \ref{8-thm-StriRod-Weig2}; Lastly we provide an application of the Strichartz estimates in Section 6.

\section{Generalized Strichartz Estimates}

In this section we prove Theorem~\ref{9-thm-Stri-Rod}, from which we see how the
generalized Strichartz estimates can be obtained from the weighted Strichartz estimates.
We also prove Theorem~\ref{9-thm-Stri-Rod-local} which illustrates that local in time
generalized Strichatz estimates can be obtained from the Morawetz-KSS estimates.

\subsection{Generalized Strichartz Estimates}\label{sec-2.1}

Now we prove Theorem~\ref{9-thm-Stri-Rod} by using weighted
Strichartz estimate~\eqref{9-est-Stri-FW6} and Rodnianski's argument
(see~\cite{St05}).

Let $f_{1,N}$ be a unit frequency function of angular frequency $N$
and $u_{1,N}=e^{i t \D^a} f_{1,N}$ with $a>0$. Denote the norm
$$
\|f\|_{\ell^{p} L^{q}_Q} = \Big (\sum_{\alpha}
\|f\|_{L^q(Q_\alpha)}^p\Big )^{\frac 1p},
$$
where $\{Q_\alpha\}$ is a partition of ${\Bbb R}^n$ into cubes
$Q_\alpha$ of side length $1$.

First, since $q\geq 2$, by using the Sobolev embedding $ \<N\>^{-\frac {n-1}q}
L^q_\omega \subset L^\infty_\omega$ on the unit sphere ${\Bbb
S}^{n-1}$ for angular frequency localized functions (see Lemma \ref{lem-7.2} in Appendix \ref{sec-app}),
 we have the following estimate
for any tiling of $\mathbb{R}^n$ by cubes $\{Q_\alpha\}$ of side
length $1$:
\begin{align}
    \|u_{1,N}(t)\|_{L^q(Q_\alpha)} &\les
    ( \|1\|_{L^q_\omega(|x|\omega\in Q_\alpha)}
    \|{u_{1,N}(t)}\|_{L^\infty_\omega} )_{L^q_{|x|}} \nonumber\\
    &\lesssim  \
     \ \|
    \<x\>^{-\frac{n-1}{q}}\, \<N\>^{\frac{n-1}{q}} u_{1,N}(t)  \|_{L^q_x(Q_\al)}
    \ , \nonumber
\end{align}
where we have used the fact that
$$|\{\omega: |x|\omega\in Q_\alpha\}|\les
(1+|x|)^{-(n-1)}.$$

This means that we have
\begin{equation}
\|u_{1,N}(t)\|_{\ell^\infty L^q_Q} \les
 \|\<x\>^{-\frac{n-1}{q}}   \<N\>^{\frac {n-1}{q}} u_{1,N}(t) \|_{L^q_x}.
\end{equation}
Interpolating this with the trivial  estimate
$$
\|u_{1,N}(t)\|_{\ell^q L^q_Q} \,\les\,
 \ \|u_{1,N}(t)\|_{L^q_x},
$$
 we arrive at the following
estimate for $r\geq q$:
\begin{equation}
    \|u_{1,N}\|_{L^q_t \ell^{r} L^q_Q}
    \ \lesssim \
    \|\<x\>^{-d} \<N\>^{d} u_{1,N}\|_{L_t^q L_x^q} \ , \label{60-est-igor}
\end{equation}
where $d=(n-1)(\frac{1}{q}-\frac{1}{r})$.

Recall the weighted Strichartz estimates \eqref{9-est-Stri-FW6},
i.e., \beeq\||x|^{\frac{n}{2}-\frac{n}{q}-\frac{b}{2}} e^{i t \D^a}
f (x) \|_{L^q_{t, |x| } L^2_\omega } \les
\|\D^{\frac{b}{2}-\frac{a}{q}}\La_\omega^{\frac{1-b}{2}}f\|_{L^2_x},
\eneq  if $b\in (1,n) $, $a>0$ and $q\in [2,\infty]$. By Sobolev
embedding on the sphere $\Sp^{n-1}$ (see \eqref{7.4} in Appendix \ref{sec-app}), we
have
\beeq\label{ineq-angu}\||x|^{\frac{n}{2}-\frac{n}{q}-\frac{b}{2}}
e^{i t \D^a} f_{1,N} (x) \|_{L^q_{t, x}} \les
\|\<N\>^{\frac{n-b}{2}-\frac{n-1}{q}}f_{1,N}\|_{L^2_x}.
\eneq

Combining \eqref{60-est-igor} and \eqref{ineq-angu}, we have
\beeq\label{60-est-Stri-generalized2}\|e^{i t \D^a} f_{1,N}\|_{L^q_t
\ell^{r} L^q_Q} \les
\|\<N\>^{1/q} f_{1,N}\|_{ L_x^2}\sim
   \|\La_\omega^{1/q} f_{1,N}\|_{ L_x^2},\eneq
if $1/q<(n-1)(1/2-1/r)$, $2\le q\le r$.

We can see that \eqref{60-est-Stri-generalized2} allows a wider
range of $(q,r)$ than that in the usual Strichartz estimates, i.e.,
we have \beeq\label{60-est-Stri-generalized}\|e^{i t \D^a}
f_{1,N}\|_{L^q_t L^r_x} \les \|\La_\omega^{1/q} f_{1,N}\|_{
L_x^2}.\eneq To see this, we compute, using the Sobolev embedding in
$Q_\alpha$, that for any $q\le r$: \begin{eqnarray*}
   \|e^{i t
\D^a} f_{1,N}\|_{L^q_t L^r_x} &=&
    \|e^{i t \D^a} f_{1,N}\|_{L^q_t \ell^r_\alpha L^r(Q_\alpha)}\\
&    \les& \|e^{i t \D^a} (1-\Delta)^n
f_{1,N}\|_{L^q_t \ell^r_\alpha L^q(Q_\alpha)}\\
&    \les& \|\La_\omega^{1/q}  (1-\Delta)^n f_{1,N}\|_{ L_x^2}\\
&    \les& \|\La_\omega^{1/q}   f_{1,N}\|_{ L_x^2}\
    .
    \end{eqnarray*}

Now we do the Littlewood-Paley decomposition $f=\sum_{j\in\Z}f_j$,
where $\hat{f_j}(\xi)=(\varphi(\xi/2^j)-\varphi(\xi/2^{j-1}))\hat
f(\xi)$ for some real-valued radially symmetric bump function
$\varphi(\xi)$ adapted to $\{\xi\in\R^n:|\xi|\leq 2\}$ which equals
1 on the unit ball. Furtherly we make a spherical decomposition of each $f_j$ and let $f_j=\sum_{N\in 2^\mathbb{N}\cup 0}f_{jN}$, where $f_{jN} $ has angular frequency $N$. By using Littlewood-Paley-Stein theorem (see Theorem 2 \cite{Stri72}) and applying \eqref{60-est-Stri-generalized}, we get for $r<\infty$
\begin{eqnarray*}
\|e^{itD^a}f\|_{L^q_t L^r_x}&\simeq &\| e^{itD^a}f_{jN}\|_{L^q_tL^r_x l^2_{j,N}}\\
&\les &\| e^{itD^a}f_{jN}\|_{l^2_{j,N} L^q_tL^r_x }\\
&\les& \|2^{js}\La_\omega^{{\frac 1
q}}f_{jN}\|_{l^2_{j,N}L^2_x} \hspace{0.2in}
\text{where} \;\;s=\frac {n}2-\frac aq-\frac {n}{r}\\
&\les& \|2^{js}\La_\omega^{{\frac 1 q}}f_j\|_{l_j^2L_x^2}\\
&\simeq& \|\La_\omega^{{\frac 1 q}}f\|_{\dot H^s}.
\end{eqnarray*}
This completes the proof of Theorem \ref{9-thm-Stri-Rod} for $2\leq q\le
r$ and $r<\infty$. The case when $q\geq r$ comes from
interpolation with the energy estimate with $(q,r)=(\infty,2)$.

\subsection{Local in Time Strichartz Estimates for the Wave Equation}
In this subsection, we prove Theorem \ref{9-thm-Stri-Rod-local}.

When considering the wave equation, if we denote
$$A_\mu(T)=\left\{\begin{array}{ll}1,& \mu>1/2,\\\log(2+T)^{\frac{1}{2}},&
\mu=1/2,\\T^{\frac{1}{2}-\mu},& 0\le\mu<1/2,
\end{array}\right.$$
then the Morawetz-KSS estimates can be stated
as \beeq\label{8-est-KSS-gene}\| \langle x\rangle^{-\mu} e^{i t \D}
f\|_{L_{[0,T]}^2 L^2_{x}}\les  A_\mu (T)\|f\|_{L^2_x}.\eneq

For the sake of completeness, we present the proof of the Morawetz-KSS estimates in Appendix \ref{app.2}.

We consider now the remaining case $1/q\ge (n-1)(1/2-1/r)$ and
$q,r\ge 2$. We will apply the Morawetz-KSS estimates for the wave equation to conclude some
local in time generalized Strichartz estimates for $a=1$.

 \textbf{Proof of Theorem \ref{9-thm-Stri-Rod-local}.}
Set $q=2$ in \eqref{60-est-igor}, if $d=1/2$ (and hence
$r=2\frac{n-1}{n-2}$), we have by \eqref{8-est-KSS-gene}
\begin{eqnarray*}
\|e^{i t \D} f_{1,N}\|_{L^2_T L^{r}_x}&=&\|e^{i t \D} f_{1,N}\|_{L^2_T  \ell^r L^{r}_Q}\\
&\les& \|e^{i t \D} (1-\Delta)^n
f_{1,N}\|_{L^2_T \ell^{r} L^2_Q}\\
 &\les&
    \|\<x\>^{-1/2}N^{1/2} e^{i t \D} (1-\Delta)^n f_{1,N}\|_{L^2_T L_x^2}\\
    &\les&
(\ln(2+T))^{1/2}
    \|N^{1/2} (1-\Delta)^n f_{1,N}\|_{ L_x^2}\\
    &\les&
(\ln(2+T))^{1/2}
    \|\La_\omega^{1/2} f_{1,N}\|_{ L_x^2}\ .
\end{eqnarray*}
Interpolating with the energy estimates, we have
\beeq\label{60-est-Stri-g-bdry}\|e^{i t \D} f_{1,N}\|_{L^q_T
L^{r}_x} \les  (\ln(2+T))^{1/q}
    \|\La_\omega^{1/q} f_{1,N}\|_{ L_x^2},\eneq
for $1/q = (n-1)(1/2-1/r)\le 1/2$.

By rescaling, we have that for any $\la>0$, \begin{eqnarray*}
  \|e^{i t \D} f_{\la,N}\|_{L^q_T
L^{r}_x} &\les&  (\ln(2+\la T))^{1/q}\la^{1/2-1/r}
    \|\La_\om^{1/q} f_{\la,N}\|_{ L_x^2}\\
    &\les& (\ln(2+T))^{1/q}(\ln(2+\la))^{1/q}\la^{1/2-1/r}
    \| \La_\om^{1/q} f_{\la,N}\|_{ L_x^2}.
\end{eqnarray*} Then \eqref{8-est-StriRod-bdry} and \eqref{8-est-StriRod-bdryenpt} come from the Littlewood-Paley decompostion.

For the case $(n-1)(1/2-1/r)<1/q \le 1/2$, we first set $q=2$ in
\eqref{60-est-igor}, then for $d<1/2$ (and hence $r<2\frac{n-1}{n-2}$) we
have
$$\|e^{i t \D} f_{1,N}\|_{L^2_T \ell^{r} L^2_Q}
\les
    \|\<x\>^{-d}N^{d} e^{i t \D} f_{1,N}\|_{L^2_T L_x^2} \les
T^{1/2-d}
    \| \La_\omega^{d} f_{1,N}\|_{ L_x^2}.$$
Interpolating with the energy estimates, we get that
\beeq\label{60-est-Stri-g-local}\|e^{i t \D} f_{1,N}\|_{L^q_T
L^{r}_x} \les T^{1/q-(n-1)(1/2-1/r)}
    \|\La_\omega^{(n-1)(1/2-1/r)} f_{1,N}\|_{ L_x^2},\eneq
for $(n-1)(1/2-1/r)<1/q \le 1/2$. Again by rescaling and the Littlewood-Paley
inequality we get
\eqref{8-est-StriRod-local}.

\section{$L^q L^r L^2_\omega$ Generalized Strichartz Estimates}

In this section, we give the proof of Theorem \ref{8-thm-StriSt}, inspired by the recent work of Smith, Sogge and Wang \cite{SmSoWa09} and Fang and Wang \cite{FaWa5}.

We shall show that
\begin{equation}\label{GS-a}
\|e^{itD}f\|_{L^q_t L^r_{|x|}L^2_\omega({\mathbb R}\times {\mathbb
R}^n)} \le C_{q,r,n} \|f\|_{L^2({\mathbb R}^n)}
\end{equation}
if $q\ge 2$, $1/q<(n-1)(1/2-1/r)$ or
$(q,r)=(\infty,2)$, and $\Hat f$ is supported in $\{\xi: |\xi|\in[1/2,1]\}$. We shall
also prove that
\begin{equation}\label{GS-a1}
\|e^{itD}f\|_{L^q_T L^r_{|x|}L^2_\omega({\mathbb R}\times {\mathbb
R}^n)} \le C_{q,r}(T) \|f\|_{L^2({\mathbb R}^n)}
\end{equation} for any $q,r\ge 2$, $1/q\geq (n-1)(1/2-1/r)$, $f$ such that
$\supp\Hat f\subset \{\xi: |\xi|\in[1/2,1]\}$ and
$$C_{q,r}(T)=\left\{\begin{array}{ll}
C_{q,r,n} T^{\frac{1}{q}-(n-1)(\frac{1}{2}-\frac{1}{r})}& \frac{1}{q}>(n-1)(\frac{1}{2}-\frac{1}{r})\\
C_{q,r,n} (\ln(2+T))^{\frac{1}{q}}& \frac{1}{q}=(n-1)(\frac{1}{2}-\frac{1}{r})
\end{array}\right.
$$

By scaling, Littlewood-Paley theory and interpolation, we get from
\eqref{GS-a} that if we remove the support assumptions on the Fourier
transform, then
\begin{equation}\label{GS-b}
\|e^{itD}g\|_{L^q_tL^r_{|x|}L^2_\omega({\mathbb R}\times {\mathbb
R}^n)} \le C_{q,r,n}\|g\|_{\dot H^{n(1/2-1/r)-1/q}({\mathbb R}^n)}.
\end{equation}
for $(q,r,n)$ satisfying
$$\tfrac1q <(n-1)(\tfrac12-\tfrac1r) \textrm{ or }(q,r)=(\infty,2), q\ge 2,
(q,r)\neq (2,\infty), (q,r)\neq (\infty,\infty).$$

At first, we use scaling and Littlewood-Paley decomposition to
conclude from \eqref{GS-a} that we have
\begin{equation}\label{GS-b1}
\|e^{itD}g\|_{L^q_tL^r_{|x|}L^2_\omega({\mathbb R}\times {\mathbb
R}^n)} \le C_{q,r}\|g\|_{\dot B^{n(1/2-1/r)-1/q}_{2,1}({\mathbb
R}^n)}.
\end{equation}

Recall that we have the interpolation between spaces of vector-valued
functions (see 5.8.6 of \cite{BeLo} page 130 or Theorem 1.18.4 of
\cite{Tr78} page 128)
\begin{equation}\label{GS-interpolation1}(L^{q_0}(A_0),
L^{q_1}(A_1))_{\theta,q}=L^q((A_0,A_1)_{\theta,q})\end{equation} if
$\frac 1q=\frac{1-\theta}{q_0}+\frac{\theta}{q_1}$, $1\le
q_0,q_1<\infty$ and $\theta\in (0,1)$, and the interpolation of the
homogeneous Besov spaces (see Theorem 6.4.5 of \cite{BeLo} page 152)
\begin{equation}\label{GS-interpolation2}(\dot B^{s_0}_{2,1},\dot B^{s_1}_{2,1})_{\theta,q}
=\dot B^{s}_{2,q}\supset\dot B^{s}_{2,2}=\dot H^s\end{equation} if
$q\ge 2$.

Based on \eqref{GS-b1}, \eqref{GS-interpolation1} and
\eqref{GS-interpolation2} for fixed $r\in (2,\infty]$, we get for
$q$ with $0<1/q<(n-1)(1/2-1/r)$ and $q>2$,
\begin{equation}\label{GS-b2}
\|e^{itD}g\|_{L^q_tL^r_{|x|}L^2_\omega({\mathbb R}\times {\mathbb
R}^)} \le C_{q,r}\|g\|_{\dot B^{n(1/2-1/r)-1/q}_{2,q}({\mathbb
R}^n)}\le C_{q,r}\|g\|_{\dot H^{n(1/2-1/r)-1/q}({\mathbb R}^n)}.
\end{equation}

This gives us the result \eqref{GS-b} for $2<q<\infty$. For the case
$q=\infty$ and $r<\infty$, the result is just the consequence of
energy estimates and Sobolev embedding. To prove the remaining case
with $q=2$ and $r<\infty$, we need only to use the fact that
$$(L^{r_1}_{|x|}L^2_\omega,L^{r_2}_{|x|}L^2_\omega)_{\theta,r}=L^r_{|x|}L^2_\omega$$
for $1\le r_1, r_2<\infty$, $\theta/r_1+(1-\theta)/r_2=1/r$ and
$\theta\in (0,1)$ (see Theorem 1.18.5 of \cite{Tr78} page 130).

 This concludes the proof of
\eqref{GS-b} for $(q,r)\neq (2,\infty), (\infty,\infty)$ with
$1/q<(n-1)(1/2-1/r)$ and $q\ge 2$. Also we can get \eqref{8-est-StriSt-2D2} and \eqref{8-est-StriSt-bdry} from \eqref{GS-a1} by the same argument.

\begin{rem} All of the requirements for $(q,r)$ are necessary for the estimates except the
requirement $(q,r)\neq (2,\infty)$. In general, we expect that
this restriction can be relaxed. In
particular, when $n=3$, the estimate with $(q,r)=(2,\infty)$ are
proven to be true in Machihara,  Nakamura, Nakanishi, and Ozawa
\cite{MaNaNaOz05}. However, we will not exploit this issue.
\end{rem}

\subsection{Proof of (3.1)}

To begin, let us recall some basic knowledge about the spherical
harmonics (for detailed discussion, see e.g. Stein and Weiss
\cite{SW71}). Let $n\ge 2$. For any $k\ge 0$, we denote by
$\mathcal{H}_k$ the space of spherical harmonics of degree k on
$\Sp^{n-1}$, by $d(0)=1$ and
$d(k)=\frac{2k+n-2}{k}C^{n+k-3}_{k-1}\simeq \langle k\rangle^{n-2}$
(for $k\ge 1$) its dimension, and by $\{Y_{ k, 1} , \cdots , Y_{ k,
d(k)} \}$ the orthonormal basis of $\mathcal{H}_k$. It is well known
that $L^2(\Sp^{n-1}) = \bigoplus_{k=0}^\infty \mathcal{H}_k$ and
that $F(t, x) = F(t, r \om)$ has the expansion
\beeq\label{GS-Exp-SpHm}F(t, r\om)=\sum_{k=0}^\infty
\sum_{l=1}^{d(k)}  a_{k,l}(t,r) Y_{k,l}(\omega).\eneq By
orthogonality, we observe that $\|F(t,
r\cdot)\|_{L^2_\om}=\|a_{k,l}(t,r)\|_{l^2_{k,l}}$.

Due to the support assumptions for the Fourier inversion of $f$
(denoted by $\check f$) we have that
\begin{equation}\label{GS-c}
\|f\|^2_{L^2({\mathbb R}^n)}\approx \int_0^\infty
\int_{\mathbb{S}^{n-1}}|\check f(\rho \omega)|^2 d\omega d\rho\ .
\end{equation}
If we expand the angular part of $\check f$ using spherical
harmonics, we find that if $\xi = \rho\omega$ with $\omega\in
\Sp^{n-1}$, then there are generalized Fourier coefficients
$c_k(\rho)$ which vanish when $\rho\notin [1/2,1]$ so that
$$\check f(\xi)=\sum_{k,l} c_{k,l}(\rho)Y_{k,l}(\omega)\ .$$ So, by \eqref{GS-c}
and Plancherel's theorem for $\Sp^{n-1}$ and ${\mathbb R}$ we have
\begin{equation}\label{GS-d}
\|f\|^2_{L^2({\mathbb R}^n)} \approx \sum_{k,l}\int_{\mathbb
R}|c_{k,l}(\rho)|^2\, d\rho \approx \sum_{k,l}\int_{{\mathbb R}} |\check
c_{k,l}(s)|^2 \, ds,
\end{equation}
if $\check c_{k,l}$ denotes the one-dimensional Fourier inversion of
$c_{k,l}(\rho)$.

Recall that (see \cite{SW71} Chapter IV Theorem 3.10 page 158)
\beeq\label{GS-F-SpHm}
  \widehat{c_{k,l}(\rho) Y_{k,l}(\omega)}(x) =
  g_{k,l}(|x|)Y_{k,l}(\frac{x}{|x|})\ ,
\eneq where \beeq\label{GS-gkl-from-fkl}
  g_{k,l}(r)=(2 \pi)^{\frac{n}{2}} i^{- k} r^{-\frac{n-2}{2}}  \int^\infty_0 c_{k,l}(\rho)
        J_{k+\frac{n-2}{2}}(r \rho) \rho^{\frac{n}{2}} d \rho\ ,\eneq
and $J_m$ is the $m$-th Bessel function with $m\in \frac12\Z$ and
$m\ge 0$, we see that we have the formula (with $x=r\vartheta$ and
$\vartheta\in \Sp^{n-1}$)
\begin{equation}\label{GS-e}
f(r\vartheta)=(2\pi)^{n/2}\sum_{k,l}  i^{- k} r^{-\frac{n-2}{2}}
\int^\infty_0 c_{k,l}(\rho)
        J_{k+\frac{n-2}{2}}(r \rho) \rho^{\frac{n}{2}} d \rho\
Y_{k,l}(\vartheta)\ .
\end{equation}
We will use two integral representations of $J_m$ as
follows. The first is Schl\"{a}fli's generalization of Bessel's
integral (see Section 6.2 (4) page 176 in \cite{Wa44})
\begin{equation}\label{GS-f}
J_m(y)=\frac{e^{-im\pi/2}}{2\pi}\int_{0}^{2\pi}
e^{iy\cos\theta-im\theta}\, d\theta-\frac{\sin
(m\pi)}{\pi}\int_0^\infty e^{-y\sinh u-mu}\,d\theta .
\end{equation}
The second is Lommel's expression of Bessel function (see
Section 3.3 (1) page 47 in \cite{Wa44})
\begin{equation}\label{GS-f2}
J_m(y)=\frac{(y/2)^m}{\sqrt\pi\ \Gamma(m+1/2)}\int_{-1}^{1}
e^{iyt}(1-t^2)^{m-\frac 12}\, dt\ .
\end{equation}

Because of \eqref{GS-e} and the support properties of the $c_{k,l}$,
we find that if we fix $\eta\in C^\infty_0({\mathbb R})$ satisfying
$\eta(\tau)=1$ for $1/2\le \tau\le 1$ and $\eta(\tau)=0$ for
$\tau\notin [1/4,2]$ and if we set $\alpha(\rho)=\rho^{n/2}
\eta(\rho) \in {\mathcal S}({\mathbb R})$, then we have
$$
(e^{itD}f)(r\vartheta)=(2\pi)^{n/2}\sum_{k,l}  i^{- k}
r^{-\frac{n-2}{2}} I_{k,l}(t,r) Y_{k,l}(\vartheta),$$ where if we
apply \eqref{GS-f},
\begin{align*}&I_{k,l}(t,r)\\
&= \int^\infty_0 c_{k,l}(\rho)
        J_{k+\frac{n-2}{2}}(r \rho) e^{it\rho} \rho^{\frac{n}{2}} d
        \rho\\
&=\int_0^\infty \int_{-\infty}^\infty J_{k+\frac{n-2}{2}}(r\rho)
e^{i\rho(t-s)} \check c_{k,l}(s)\alpha(\rho)\, ds d\rho\,
\\
&=\frac{e^{-i(k+\frac{n-2}2)\pi/2}}{2\pi}\int_0^\infty
\int_{-\infty}^\infty
 e^{i\rho(t-s)} \check c_{k,l}(s)\alpha(\rho)\, \int_{0}^{2\pi}
e^{i\rho r\cos\theta -i (k+\frac{n-2}2)\theta}\, d\theta ds d\rho
\\
&\;\;-\frac{\sin((k+\frac{n-2}2)\pi)}{\pi}\int_0^\infty
\int_{-\infty}^\infty
 e^{i\rho(t-s)} \check c_{k,l}(s)\alpha(\rho)\, \int_{0}^{\infty}
e^{-\rho r\sinh u-(k+(n-2)/2)u}\, du ds d\rho
\\
&=
e^{-i(k+\frac{n-2}2)\pi/2} \int_{-\infty}^\infty \int_{0}^{2\pi}
 \check c_{k,l}(s)\check\alpha(t-s+r\cos\theta)\,
e^{-i (k+\frac{n-2}2)\theta}\, d\theta ds
\\
&\;\; -2\sin((k+\frac{n-2}2)\pi)
\int_{-\infty}^\infty \int_{0}^{\infty}
 \check c_{k,l}(s)e^{-(k+(n-2)/2)u}\check\beta(t-s,r,u)\,
\, du ds,
\end{align*}
where $\beta (\rho,r,u)=\al(\rho)e^{-\rho r\sinh u}$ and the inverse
Fourier transformation acts on the first variable of $\beta$. And if
we apply \eqref{GS-f2} instead of \eqref{GS-f}, then
\begin{eqnarray*}
&& I_{k,l}(t,r)=\frac 1{2^{k+\frac{n-2}2}\sqrt\pi\
\Gamma(k+\frac{n-1}2)}
\\
&&
\times\int_0^\infty \int_{-\infty}^\infty
e^{i\rho(t-s)} \check c_{k,l}(s)\alpha(\rho)(\rho
r)^{k+\frac{n-2}2}\int_{-1}^1e^{i\rho ru}(1-u^2)^{k+\frac{n-3}2}\, duds d\rho\,
\\
&&=2\pi \frac {r^{k+\frac{n-2}2}}{2^{k+\frac{n-2}2}\sqrt\pi\
\Gamma(k+\frac{n-1}2)}\int_{-\infty}^\infty \int_{-1}^{1}
 \check c_{k,l}(s)\check\gamma(t-s+ru)(1-u^2)^{k+\frac{n-3}2}\,
\, du ds,
\end{eqnarray*}
where $\gamma(\rho)=\al(\rho)\rho^{k+\frac{n-2}2}$. For simplicity,
we introduce new functions $\psi_{ik}(m,r)$ ($i=1,2,3$)
\beeq\label{GS-eq-def}\psi_{1k}(m,r)=\int_{0}^{2\pi}
e^{-i(k+\frac{n-2}2)\theta}\check\alpha\bigl(\, m+r\cos\theta \,
\bigr)\, \, d\theta ,\eneq
\beeq\label{GS-eq-def2}\psi_{2k}(m,r)=\sin((k+\frac{n-2}2)\pi)\int_{0}^{\infty}
e^{-(k+\frac{n-2}{2})u}\check\be\bigl(\, m,r,u\, \bigr)\, \, du
,\eneq \beeq\label{GS-eq-def3}\psi_{3k}(m,r)=\frac
{2\pi}{2^{k+\frac{n-2}2}\sqrt\pi\ \Gamma(k+\frac{n-1}2)}
r^{k+\frac{n-2}2}\int_{-1}^{1}
 \check\gamma(m+ru)(1-u^2)^{k+\frac{n-3}2}\,
\, du .\eneq Thus
\begin{eqnarray*}I_{k,l}(t,r)&=&e^{-i(k+\frac{n-2}2)\pi/2}\int_{\R}
\check c_{k,l}(s)\psi_{1k}(t-s,r) ds-2 \int_{\R}\check
c_{k,l}(s)\psi_{2k}(t-s,r) ds\\
&=&\int_{\R}\check c_{k,l}(s)\psi_{3k}(t-s,r) ds\ . \end{eqnarray*}
As a result, we have that for any $r> 0$,
\begin{equation}\label{GS-g}
\|(e^{itD}f)(r\vartheta)\|_{L^2_\vartheta}\les
  \sum_{i=1,2}\left\|
\int_{\R} \check c_{k,l}(s)\psi_{ik}(t-s,r) r^{-\frac{n-2}{2}}
ds\right\|_{l^2_{k,l}}, \end{equation}
and
\begin{equation}
\|(e^{itD}f)(r\vartheta)\|_{L^2_\vartheta}\les
  \left\|
\int_{\R} \check c_{k,l}(s)\psi_{3k}(t-s,r) r^{-\frac{n-2}{2}}
ds\right\|_{l^2_{k,l}}.
\end{equation}
Now we claim that we have the following estimates
\beeq\label{GS-KeyStep}\|\psi_{ik}(m,r) \left<m\right>^{\frac
{n-1}2} r^{-\frac{n-2}2}\|_{L_m^2}\le C\, , \text{for}\ i=1,2
~\text{and}~ r>1,\eneq \beeq\label{GS-KeyStep-2}\|\psi_{1k}(m,r)
\left<m\right>^{\frac {n-1}2} r^{-\frac{n-2}2}\|_{L_m^2}\le C\, ,
\text{for}\ n \textrm{ even and }r\le 1,\eneq
 \beeq\label{GS-KeyStep3}\|\psi_{3k}(m,r)
\left<m\right>^{\frac {n-1}2} r^{-\frac{n-2}2}\|_{L_m^2}\le C\, ,
\text{for}\ n \textrm{ odd and }r\leq 1,\eneq where
$\left<m\right>=\sqrt{1+m^2}$ and $C$ is independent of $k\in \Z$
and $r>0$.
Based on these estimates, we have
\begin{eqnarray*}
&&\|(e^{itD}f)(r\vartheta)\|_{L^2_\vartheta}\\
&\les&
\sum_{i=1}^2\| \check c_{k,l}(s)\psi_{ik}(t-s,r) r^{-\frac{n-2}{2}} \|_{l_{k,l}^2 L^1_s}\\
&\les& \sum_{i=1}^2 \|\check c_{k,l}(s)
\left<t-s\right>^{-(n-1)/2}\|_{l_{k,l}^2
L^2_s}\|\left<t-s\right>^{\frac {n-1}2} r^{-\frac{n-2}{2}} \psi_{ik}(t-s,r) \|_{L_s^2}\\
&\les& \|\check c_{k,l}(s) \left<t-s\right>^{-(n-1)/2}\|_{l_{k,l}^2
L^2_s}\
\end{eqnarray*} for $r>1$. For the estimates with $r\le 1$, we need only to use the same argument with the
observation that $\psi_{2k}=0$ when $n$ is even. In summary, these
estimates tell us that
$$\|(e^{itD}f)(x)\|_{L^\infty_{|x|} L^2_\omega}\le C \|\check c_{k,l}(s)
\left<t-s\right>^{-(n-1)/2}\|_{l_{k,l}^2 L^2_s}\ .$$

Recall that the energy estimates and \eqref{GS-d} tell us that
$$\|(e^{itD}f)(x)\|_{L^2_{|x|} L^2_\omega}\le \|f\|_{L^2}\le C \|\check c_{k,l}(s)
\|_{l_{k,l}^2 L^2_s}\ .$$ By interpolation, we can immediately get
\beeq\label{GS-x}\|(e^{itD}f)(x)\|_{L^r_{|x|} L^2_\omega}\le C
\|\check c_{k,l}(s) \left<t-s\right>^{-(n-1)(1/2-1/r)}\|_{l_{k,l}^2
L^2_s}\ .\eneq Now we see that we have the estimates \eqref{GS-a},
for any $q\ge 2$ and $1/q<(n-1)(1/2-1/r)$ or $(q,r)=(\infty,2)$. In
fact,
\begin{eqnarray*}
  \|(e^{itD}f)(x)\|_{L^q_t L^r_{|x|} L^2_\omega}&\le& C \|\check
c_{k,l}(s) \left<t-s\right>^{-(n-1)(1/2-1/r)}\|_{L^q_t l_{k,l}^2
L^2_s}\\
&\le& C \|\check c_{k,l}(s)
\left<t-s\right>^{-(n-1)(1/2-1/r)}\|_{l_{k,l}^2 L^2_s L^q_t }\\
&\le& C \|\check c_{k,l}(s) \|_{l_{k,l}^2 L^2_s }\simeq \|f\|_{L^2}\
.
\end{eqnarray*} Similarly, we can also prove \eqref{GS-a1} for $1/q\geq (n-1)(1/2-1/r)$.

\subsection{The estimates for $\psi_{ik}(m,r)$} Now we present the proof
of the key estimates \eqref{GS-KeyStep}-\eqref{GS-KeyStep3} for
$\psi_{ik}(m,r)$, to conclude the proof of \eqref{GS-a}.

At first, we observe that the estimate \eqref{GS-KeyStep} for
$\psi_{1k}$ has been obtained in \cite{SmSoWa09} and \cite{FaWa5}.
Moreover, $\psi_{2k}=0$ when $n$ is even. Thus we need only to prove
\eqref{GS-KeyStep} for $\psi_{2k}$ with $n\ge 3$ odd,
\eqref{GS-KeyStep-2} for $\psi_{1k}$ and \eqref{GS-KeyStep3} for
$\psi_{3k}$.

\begin{lem}\label{GS-mainest}  Let $n\ge 3$ be odd, $\be(\rho,r,u)=\al(\rho)e^{-\rho r \sinh
u}$ and $r>1$. Then there is a uniform constant $C$, which is
independent of $k$ and $r>1$ so that the following inequality hold
\beeq\label{GS-3}\|\psi_{2k}(m,r) \left<m\right>^{\frac {n-1}2}
r^{-\frac{n-2}2}\|_{L_m^2}\le C\, . \eneq
\end{lem}

\begin{proof}
Notice
that $\be\in \mathcal{S}(\R)$ with respect to all variables (together
with the support in $[1/4,2]$ for $\rho$). If we use
H\"{o}lder's inequality, Plancherel Theorem and the facts $r>1$ and
$k+\frac{n-2}2\ge \frac 12$, then
\begin{eqnarray*}
\|\psi_{2k}(m,r) \left<m\right>^{\frac {n-1}2}
r^{-\frac{n-2}2}\|_{L_m^2}&=&\|\left<m\right>^{\frac {n-1}2}
r^{-\frac{n-2}2}\int_0^\infty \check
\be(m,r,u)e^{-(k+\frac{n-2}2)u}\,du \|_{L_m^2}\\
&\les&\|\check \be(m,r,u)\left<m\right>^{\frac {n-1}2}
\|_{L_u^2L_m^2}\\
&\les&\|\be(\rho,r,u)\|_{L^2_u H^{\frac{n-1}2}_\rho}\\
&\les&\sum_{0\le l\le \frac{n-1}2}\|(r\sinh u)^{\frac {n-1}2-l}e^{-\rho r\sinh u}\al^{(l)}(\rho)
\|_{L^2_\rho L^2_u}\\
&\les&\|\<r\sinh u\>^{ \frac{n-1}2}e^{-\frac{r\sinh u}{2}}\|_{L^2_u}\\
&\le& C.
\end{eqnarray*}
\end{proof}

\begin{lem}
If $n\ge 2$ is even, then for any $r\in (0,1]$ and $N>0$, we have
\begin{equation}\label{GS-3-2}
|\psi_{1k}(m,r)|\le C_N r^{(n-2)/2}  \<m\>^{-N}.
\end{equation} Consequently, we have \eqref{GS-KeyStep-2} for
$\psi_{1k}$.
\end{lem}
\begin{prf}
First, observe that the case $n=2$ is trivial since $\al\in
\mathcal{S}$. For the case $n\ge 4$ and even, then $\frac{n-2}2\in
\mathbb{N}$. The estimate \eqref{GS-3-2} follows immediately, if we
use the Taylor expansion of $\check \alpha$ up to order $n/2$, in
terms of $r\cos\theta$, and recall that we have the orthogonality
relation
$$\int_0^{2\pi} e^{i(k+\frac{n-2}2)\theta}(\cos\theta)^j=0,\ {\textrm{ if }}0\le j<\frac{n-2}2
\Leftrightarrow 0\le j\le \frac{n-4}2\ .$$
\end{prf}

\begin{lem}
If $n\ge 3$ is odd, the estimate \eqref{GS-KeyStep3} holds for
$\psi_{3k}$.
\end{lem}
\begin{prf}
Notice that $r<1$ and
$\ga(\rho)=\al(\rho)\rho^{k+\frac{n-2}2}=\eta(\rho)\rho^{k+n-1}$,
\begin{eqnarray*}
&&\|\psi_{3k}(m,r) \left<m\right>^{\frac {n-1}2}
r^{-\frac{n-2}2}\|_{L_m^2}\\
&=&\| \frac {2\pi}{2^{k+\frac{n-2}2}\sqrt\pi\ \Gamma(k+\frac{n-1}2)}
\left<m\right>^{\frac {n-1}2} r^{k}\int_{-1}^1
\check\gamma(m+ru)(1-u^2)^{k+\frac{n-3}2}
\,du \|_{L_m^2}\\
&\les&\frac {1}{2^{k}
\Gamma(k+\frac{n-1}2)} \int_{-1}^1 \|\left<m\right>^{\frac {n-1}2}\check\gamma(m+ru) \|_{L_m^2} (1-u^2)^{k+\frac{n-3}2}\,du\\
&\le& \frac {1}{2^{k}
\Gamma(k+\frac{n-1}2)}\int_{-1}^1 \|\gamma(\rho) \|_{H_\rho^{\frac{n-1}2}} (1-u^2)^{k+\frac{n-3}2}\,du\\
&\les& \frac {1}{2^{k} \Gamma(k+\frac{n-1}2)}
\frac{(k+n-1)!}{(k+\frac{n-1}2)!}\\
&=&
\frac{(k+n-1)(k+n-2)\cdots(k+\frac{n+1}2)}{2^k(k+\frac{n-3}{2})!}
\end{eqnarray*}
When $k\ge \frac{n+1}2$, we have $k+n-j\le 2k+n-1-2j$ for $1\le j\le
\frac{n-1}2$, and we see that the last quantity is less than or
equal to $1$. For any $k<\frac{n+1}2$, the last quantity is
bounded, and this concludes the proof of \eqref{GS-KeyStep3} with
a constant independent of $k$.
\end{prf}

\section{Radial Weighted Strichartz Estimates, A Motivation}

\label{sec-RadialEst}  In this section we study the weighted Strichartz estimates for the wave equation with radial initial data and prove Theorem \ref{full_weig_str_th_sph_data}. The argument of \cite{St05} in proving Strichartz estimates of the
wave equation with radial initial data can be adapted for our purpose.

\textbf{Proof of Theorem~\ref{full_weig_str_th_sph_data}}
Let $f_{1}$ be a radially symmetric unit frequency function and $u_{1}=e^{i t \D} f_{1}$.
For a radially symmetric initial data $f_1$, we can write $u_{1}$ as the integral formula:
\begin{equation}
        e^{it \D}f_1\, (r) \ = \
        C r^{-\frac{n-2}{2}}\int_0^\infty
        e^{ i t\rho}J_{\frac{n-2}{2}}(  r\rho)
        \, \widehat{f_1}(\rho)\, \rho^\frac{n}{2} \ d\rho \ ,
        \label{sph_wave_osc_int}
\end{equation}
where $J_{\frac{n-2}{2}}(y)$ is the Bessel function of order
$\frac{n-2}{2}$ (compare \eqref{GS-e}). Then we use the well known
asymptotics for Bessel functions of relatively small order (see
\cite{Wa44}):
\begin{equation}
      J_{\frac{n-2}{2}}(y) \ =
      \begin{cases}
              \left( \frac{2}{\pi y} \right)^\frac{1}{2}\Big[
              \cos (y - \frac{n-1}{4}\pi)\cdot m_1(y) \ - \\
              \ \ \ \ \ \ \ \ \ \ \ \ \ \ \ \ \ \
              \sin (y - \frac{n-1}{4}\pi)\cdot m_2(y)\Big] \ ,
              &\hbox{ for $y \geqslant 1$ } \ , \\
              y^\frac{n-2}{2}\cdot m_3(y) \ ,
              &\hbox{ for $0 \leqslant y \leqslant 1$ } \ .
      \end{cases}\label{bessel_asym}
\end{equation}
Here the function $m_3$ is smooth, and the remaining $m_i$ have
asymptotic expansions:
\begin{align}
        m_1(y) \ &= \ \sum_k C_{1,k}y^{-2k} \ , \notag \\
        m_2(y) \ &= \ \sum_k C_{2,k}y^{-2k-1} \ , \notag
\end{align}
as $y\to \infty$. In other words, the functions $m_1( r\rho)$ and
$m_2( r\rho)$ are smooth with derivatives in $\rho$ uniformly
bounded for all $\frac{1}{2} \leqslant \rho \leqslant 2$ and $r
\geqslant 2$. Substituting the asymptotic formula \eqref{bessel_asym} into
the integral formula, we may assume without loss of generality that
we are trying to bound integrals of the form:
\begin{equation}
        I^\pm (t,r) \ = \ \frac{1}{(1 + r )^\frac{n-1}{2}}\int_{-\infty}^\infty
        e^{i (t \pm r)\rho}\, m^\pm (r,\rho)\, \chi_{(1/4,4)}(\rho)\,
        \widehat{f_1}(\rho)\ d\rho \ , \label{basic_sph_int}
\end{equation}
where $m^\pm$ is a smooth function with derivatives in $\rho$
uniformly bounded for all $r \geqslant 0 $, and $\chi_{(1/4,4)}$ is
a smooth bump function on the interval $(\frac{1}{4},4)$. It is now
apparent that the integrals in \eqref{basic_sph_int} are essentially
time translated inverse Fourier transforms of a one dimensional unit
frequency function. Therefore, we can localize these integrals in
physical space (i.e. the $t\pm r$ variable) on a $O(1)$ scale.
Since the function $\widehat{f_1}$ is compactly supported in the
interval $(0,4)$, we may take its Fourier series development: \beeq
        \widehat{f_1}(\rho) \ = \ \sum_{k=-\infty}^\infty \ c_k \, e^{i k \rho}
        ,\ \rho \ \in \ (0,4) \ . \label{f_fourier_devel}
\eneq An important thing to notice here is that we can recover the
$L^2$ norm of $f_1$ as a function on $\mathbb{R}^n$ in terms of
$\{c_k\}$:
\begin{equation}
        \lp{f_1}{L^2_x}^2 \ \sim \lp{\hat f_1}{L^2_\rho}^2\ \sim \ \sum_k\ |c_k|^2
        \ . \label{f_recover_est}
\end{equation}
Sticking the series \eqref{f_fourier_devel} into the the integrals
\eqref{basic_sph_int} yields:
\begin{equation}
        I^\pm (t,r) \ = \ \sum_k\
        \frac{c_k}{(1 + r )^\frac{n-1}{2}}\, \psi^\pm_k(t,r) \ ,
        \label{I_wavelet_exp}
\end{equation}
where
\begin{equation}
        \psi^\pm_k(t,r) \ = \ \int_{-\infty}^\infty
        e^{i (t \pm r + k)\rho}\,
        m^\pm (r,\rho)\, \chi_{(1/4,4)}(\rho)\ d\rho \ . \notag
\end{equation}
Integrating by parts as many times as necessary in the above
formula, we see that we have the asymptotic bound:
\begin{equation}
        |\psi^\pm_k(t,r)| \ \leqslant \
        \frac{C_M}{(1 + | t \pm r + k|)^{2M}},\ \forall M\in \mathbb{N} \ . \label{psi_asym}
\end{equation}
Using the expansion \eqref{I_wavelet_exp} and the asymptotic bound
\eqref{psi_asym} we can directly compute that
\begin{align}
        \lp{|x|^{-\al} I^\pm (t,|x|)}{L^p_x} \ &\lesssim
        \|c_k \psi^\pm_k(t,r) (1 + r)^{-\frac{n-1}{2}}r^{\frac{n-1}{p}-\al}\|_{L^p_r \ell^1_k}\  \notag \\
        &\lesssim \|c_k (1 + | t \pm r + k|)^{-2 M} (1 + r)^{-\frac{n-1}{2}}r^{\frac{n-1}{p}-\al}\|_{L^p_r \ell^1_k}
         \notag \\
        &\lesssim \|c_k (1 + | |t+k|- r |)^{- M} (1 + r)^{-\frac{n-1}{2}}r^{\frac{n-1}{p}-\al}\|_{L^p_r \ell^p_k}
        . \notag \end{align}
The manipulation to get the last line above follows from H\"older's
inequality. Note that to make the function integrable in $L^p$, we
must have $\al$ be the number such that $\frac{n-1}{p}-\al>-\frac
1p$ (or $-\al\ge 0$ when $p=\infty$), i.e.,
\beeq\label{9-est-3-CondOnAlpha}\al< \frac{n}{p}\ \ (\al\le
0\textrm{ if }p=\infty).\eneq
If we also choose  $M$ large enough,
then by integrating each expression in this line term by term, we
arrive at the bound:
$$ \lp{|x|^{-\al} I^\pm (t,|x|)}{L^p_x}  \lesssim \|
        c_k( 1 + | t + k | )
        ^{- (n-1)(\frac{1}{2} - \frac{1}{p})-\al}\|_{\ell_k^p}.$$
Testing this last expression for $L^q$ in time, and using the
inclusion $\ell^{\min\{p,q\}} \subseteq \ell^p$,
 we see that:
\begin{eqnarray*}
      \lp{|x|^{-\al} I^\pm (t,|x|)}{L^q_t L^p_x}   & \lesssim &
       \|c_k ( 1 + | t + k | )^{-(n-1)(\frac{1}{2} - \frac{1}{p})-\al}
       \|_{L^q_t \ell^p_k}\\
       &\lesssim &
       \|c_k ( 1 + | t + k | )^{-(n-1)(\frac{1}{2} - \frac{1}{p})-\al}
       \|_{\ell^{\min\{p,q\}}_k L^q_t }\\
       &\les& \|c_k \|_{\ell^{2}_k}
       \ ,
\end{eqnarray*}
as long as $\min\{p,q\}\ge 2$ and $\frac 1q <  (n-1)(\frac{1}{2} -
\frac{1}{p})+\al$ (or $0\leq (n-1)(\frac{1}{2} -
\frac{1}{p})+\al$ for $q=\infty$).

In conclusion, using the characterization \eqref{f_recover_est}, we
know that if we have $q,p\ge 2$ and
\beeq\label{9-est-3-CondOnAlpha2}\frac 1q -  (n-1)(\frac{1}{2} -
\frac{1}{p})<\al<\frac np\eneq (we can take the first inequality with equality
when $q=\infty$ and the second inequality with equality when $p=\infty$), then we have
\beeq\label{4.10}\||x|^{-\al}e^{it\D} f_1\|_{L^q_t L^p_x}\les \|f_1\|_{L^2}.
\eneq
Now we use the Littlewood-Paley decomposition $f=\sum_{j\in\Z}f_j$,
and apply \eqref{4.10} to get
\begin{eqnarray*}
\||x|^{-\al}e^{itD}f\|_{L^q_t L^r_x}&=&\||x|^{-\al} e^{itD}\sum_{j\in\Z}f_j\|_{L^q_tL^r_x}\\
&\les& \sum_{j\in\Z}\||x|^{-\al}e^{itD}f_j\|_{L^q_tL^r_x}\\
&\les& \sum_{j\in\Z}\left(2^{js}\|f_j\|_{L^2_x}\right) \hspace{0.2in}
\text{where} \;\;s=\frac {n}2+\al-\frac{1}{q}-\frac {n}{r}\\
&\les& \|2^{js}f_j\|_{l_j^1L_x^2}\\
&\simeq& \|f\|_{\dot B_{2,1}^s}.
\end{eqnarray*}

Next we use real interpolation to prove the estimate \eqref{full_sph_str_est}.

\textbf{Proof of Theorem \ref{full_weig_str_th_sph_data}:} The case $(q,r)=(2,2)$ follows directly from the weighted Strichartz estimates \eqref{9-est-Stri-FW6} with $r=2$ and $a=1$. If $q=\infty$ and $2\le r<\infty$, it follows directly from the weighted Hardy-Littlewood-Sobolev inequality of Stein and Weiss \cite{StWe58}
\beeq\label{60-est-HLS-weighted-Rn} \| |x|^{-\al} f\|_{L^r(\R^n)}\les \|
|x|^{\be}\D^s f\|_{L^q(\R^n)}\eneq with $1<q\le r<\infty$, $\al<n/r$,
$\be<n/q'$, $\al+\be\ge 0$, $s\in (0,n)$ and $-s+n/q+\be=-\al+n/r$.

 For fixed $2<r<\infty$ and $\al$, we can choose $2<r_1<r<r_2<\infty$ such that
$$\frac 2r=\frac{1}{r_1}+\frac 1 {r_2},\ \frac 1q-(n-1)(\frac 12-\frac 1{r_i})<\al<\frac{n}{r_i}\ .$$
Let $d\omega_i=|x|^{-\al r_i}dx, i=1,2$, from the discussion above we have
\beeq\label{inter1}
\|e^{itD}f\|_{L^{2}_tL^{r_1}_x(dw_1)}\les \|f\|_{\dot B_{2,1}^{s_1}}\eneq
\beeq\label{inter2}
\|e^{itD}f\|_{L^{2}_tL^{r_2}_x(d\omega_2)}\les \|f\|_{\dot B_{2,1}^{s_2}}\eneq
where $s_i=\frac {n}2+\al-\frac{1}{q}-\frac {n}{r_i}, i=1,2$.

Since $(L_x^{r_1}(d\omega_1), L_x^{r_2}(d\omega_2))$ is an interpolation
couple, by Theorem 1.18.4 in \cite{Tr78} we have
$$(L^{q}_t(L_x^{r_1}(d\omega_1)),L^{q}_t(L_x^{r_2}(d\omega_2)))_{\theta,2}=L^q_t((L_x^{r_1}(d\omega_1), L_x^{r_2}(d\omega_2))_{\theta,2}).$$
And from Theorem 6.4.5 in \cite{BeLo} we know:
$$(B_{p,q_0}^{s_1},B_{p,q_1}^{s_2})_{\theta,r}=B_{p,r}^s, \text{if} s_0\neq s_1, 0<\theta<1, r,q_0,q_1\geq 1, s=(1-\theta)s_1+\theta s_2.$$
Now by real interpolation between \eqref{inter1} and \eqref{inter2}, we get
\begin{eqnarray*}
\||x|^{-\al}e^{itD}f\|_{L^{q}_tL^{r}_x}&\les &\|e^{itD}f\|_{L^{q}_t((L_x^{r_1}(d\omega_1), L_x^{r_2}(d\omega_2))_{1/2,r})}\\
&\les &\|e^{itD^a}f\|_{(L^{q}_tL_x^{r_1}(d\om_1), L^{q}_tL_x^{r_2}(d\om_2))_{1/2,r}}\\
&\les&\|f\|_{(\dot B_{2,1}^{s_1},\dot B_{2,1}^{s_2})_{1/2,r}}\\
&=&\|f\|_{\dot B_{2,r}^{s}}\\
&\les&\|f\|_{\dot H^{s}}\ .
\end{eqnarray*}
This proves \eqref{full_sph_str_est} for $2<r<\infty,2\leq q\leq \infty$.

Likewise, we can prove the case when $2<q<\infty, 2\leq r\leq \infty$.
\QED


\section{Weighted Strichartz Estimates}\label{weSt}

In this section, we prove the weighted Strichartz estimates stated
in Theorem \ref{8-thm-StriRod-Weig1}, based on Rodnianski's
argument, the weighted Strichartz estimates \eqref{9-est-Stri-FW6}
and a localized version of the weighted HLS estimates.

\subsection{Localized Weighted Hardy-Littlewood-Sobolev Inequality}
Recall that we have the classical weighted Hardy-Littlewood-Sobolev
inequality \eqref{60-est-HLS-weighted-Rn}
 of Stein and Weiss \cite{StWe58}, i.e.,
$$\| |x|^{-\al} f\|_{L^r(\R^n)}\les \|
|x|^{\be}\D^s f\|_{L^q(\R^n)}$$ with $1<q\le r<\infty$, $\al<n/r$,
$\be<n/q'$, $\al+\be\ge 0$, $s\in (0,n)$ and $-s+n/q+\be=-\al+n/r$.

In particular, if we choose $\be=-\al$, the corresponding estimate
on $\R^n$ is true for $-n/q'< \al <n/r$, $1<q\le r<\infty$ and
$s=n/q-n/r$.

In this subsection, we are interested in the localized version of
the estimate with $\be=-\al$. More precisely, if we denote $B_1$ be the unit ball, and $B_2$ the ball centered at origin with radius $2$, then we aim at the proof of
the following lemma.

\begin{lem}[Localized Weighted HLS]\label{lemm-localized-w-HLS}
Let $1<q\le r<\infty$, and $-n/q'< \al <n/r$, we have the localized
version of the weighted Hardy-Littlewood-Sobolev inequality
\beeq\label{60-est-HLS-weighted1} \| |x|^{-\al} f\|_{L^r_{x\in
B_1}}\les \sum_{|\be|\le s}\| |x|^{-\al}\pa^\be f\|_{L^q_{B_2}}\eneq
 if $s$ is large enough (in fact, we
need only to choose $s=2m$ with $m>n/2$).
\end{lem}
\begin{prf}
  We will prove the result for $s=2m$ with $n/2<m\in\N$. At first,
  we observe that the proof of the estimate \eqref{60-est-HLS-weighted1}
  can be reduced to the proof of the following
  \beeq\label{60-est-HLS-intermediate1}\| |x|^{-\al}\phi f\|_{L^r}\les \| |x|^{-\al}\Lambda^{2m}
  f\|_{L^q},\eneq where $\phi\in C_0^\infty$ and $\La^{2m}=(1-\Delta)^m$ is a differential operator. In fact, if this estimate
  is true for any $f$, we can choose $\phi=1$ in
  $B_1$, $\phi\ \psi =\phi$, $\psi=0$ in $B_2^c$, and $f=\psi\
  g$. Recall that $$[\La^{2m},\psi]
  g=\sum_{|\al|=1}^{2m}\sum_{|\al|+|\be|\le 2m}
  c_{\al,\be}
  \pa^\al \psi \pa^\be g,$$ and $\supp \pa\psi\subset B_2\backslash
  B_1$. We have
  \begin{eqnarray*}
    \| |x|^{-\al} g\|_{L^r_{x\in
B_1}}&\le& \| |x|^{-\al}\phi \psi g\|_{L^r}\\
&\les& \| |x|^{-\al}\Lambda^{2m}
  (\psi g)\|_{L^q}\\
&\les &\| |x|^{-\al}\psi\Lambda^{2m}
   g\|_{L^q}+\sum_{|\be|=0}^{2m-1}\|\pa^\be g\|_{L^q_{B_2\backslash
  B_1}}\\
&\les & \| |x|^{-\al}\Lambda^{2m}
   g\|_{L^q_{B_2}}+\sum_{|\be|=0}^{2m-1}\|\pa^\be g\|_{L^q_{B_2\backslash
  B_1}}\\
  &\les & \sum_{|\be|=0}^{2m}\| |x|^{-\al}\pa^\be
   g\|_{L^q_{B_2}}\ .
   \end{eqnarray*}

So it is sufficient to prove \eqref{60-est-HLS-intermediate1}. First we write \eqref{60-est-HLS-intermediate1} in the equivalent
form \beeq\label{60-est-HLS-intermediate21}\| |x|^{-\al}\phi
\Lambda^{-2m}|x|^{\al}  f\|_{L^r}\les \|
  f\|_{L^q}.\eneq Recall that $\Lambda^{-2m} f=\mathcal{F}^{-1}(1+|\xi|^2)^{-m} \hat{f})=K\ast
  f$, and if $m>n/2$, then $K(x)=O(\<x\>^{-N})$ for any $N$. By
  introducing $\psi\in C_0^\infty$ such that
  $\psi\phi=\phi$, if $\al<n/r$, we can control
  $$\| |x|^{-\al}\phi \Lambda^{-2m}|x|^{\al}  f\|_{L^r}\les \|
|x|^{-\al}\psi\|_{L^r}  \| \phi \Lambda^{-2m}|x|^{\al}
f\|_{L^\infty} \les \|Tf\|_{L^\infty} $$ with $$Tf(x)=\int \phi(x)
K(x-y) |y|^\al f(y) dy.$$ By H\"{o}lder's inequality, we can
estimate $\|Tf\|_{L^\infty}$ by $\|f\|_{L^q}$, if we have
$$\phi(x)
K(x-y) |y|^\al\in L^\infty_x L^{q'}_y,$$ which can be easily seen to
hold if we have $\al>-n/q'$.
\end{prf}

\subsection{Rodnianski's Argument for the Weighted Estimates}\label{sec-5.2}

Let $f_{1,N}$ be a unit frequency function of angular frequency $N$
and $u_{1,N}=e^{i t \D^a} f_{1,N}$ with $a>0$.

First, using the Sobolev inequality $ \La_\omega^{-\frac {n-1}2}
L^2_\omega \subset L^\infty_\omega$ on the unit sphere ${\Bbb
S}^{n-1}$ for angular frequency localized functions (see \eqref{7.3}
in Appendix \ref{sec-app}),
 we have the following estimate
for any tiling of $\mathbb{R}^n$ by cubes $\{Q_\al\}$:
\begin{align}
    \||x|^{-\al} u(t)\|_{L^2(Q_\al)}\ &\les \
    ( \|1\|_{L^2_\omega(|x|\omega\in Q_\beta)}
    \||x|^{-\al} {u (t)}\|_{L^\infty_\omega} )_{L^2_{|x|}} \nonumber\\
    &\lesssim  \
     \ \|
    |x|^{-\frac{n-1}{2}}|x|^{-\al} \, \La_\omega^{\frac{n-1}{2}} u (t)  \|_{L^2_x}
    \ , \nonumber
\end{align}
where we have used the fact that
$$|\{\omega: |x|\omega\in Q_\al\}|\ \les\
|x|^{-(n-1)}.$$

This means that we have
\begin{equation}
\||x|^{-\al} \, u (t)\|_{\ell^\infty L^2_Q}\, \les\,
 \| |x|^{-\frac{n-1}{2}-\al}    \La_\omega^{\frac {n-1}{2}} u (t) \|_{L^2_x}.
\end{equation}
Interpolating this with the trivial  estimate
$$ \| |x|^{-\al} u (t)\|_{\ell^2 L^2_Q} \,\les\,  \ \||x|^{-\al}u (t)\|_{L^2_x}, $$
 we arrive at the following estimate ($2\le r$):
\begin{equation}\label{60-est-igor1}     \| |x|^{-\al} u \|_{L^2_t \ell^{r} L^2_Q}
    \ \lesssim  \   \| |x|^{-d-\al} \La_\omega^{d} u \|_{L_t^2 L_x^2} \ ,
\end{equation} where $d=(n-1)(\frac{1}{2}-\frac{1}{r})$.

Recall that we have the generalized Morawetz estimates \eqref{9-est-Morawetz-homo}, i.e.,
 \beeq\label{9-FW-Weighted1}\||x|^{-\frac{b}{2}} e^{i t \D^a} f (x) \|_{L^2_{t, x }}
\les \|\D^{\frac{b-a}{2}}\La_\omega^{\frac{1-b}{2}}f\|_{L^2_x},
\eneq  if $b\in (1,n) $, $a>0$. Combining \eqref{60-est-igor1} and
\eqref{9-FW-Weighted1}, we have
\beeq\label{60-est-Stri-generalized21}\||x|^{-\al} e^{i t \D^a}
f_{1,N} \|_{L^2_t \ell^{r} L^2_Q} \les
   \|\La_\omega^{{\frac 1 2}-\al} f_{1,N} \|_{ L_x^2},\eneq
if $\frac{2-n}{2}+\frac{n-1}{r}<\al<\frac{1}{2}+\frac{n-1}{r}$ and
$2\le  r$.

We can see that \eqref{60-est-Stri-generalized21} implies the
nonhomogeneous weighted Strichartz estimates, i.e., we have
\beeq\label{60-est-Stri-weighted-nonhom}\| \<x\>^{-\al} e^{i t \D^a}
f_1 \|_{L^2_t L^r_x} \les \|\La_\omega^{{\frac 1 2}-\al}
 f_1 \|_{ L_x^2},\eneq if
$\frac{2-n}{2}+\frac{n-1}{r}<\al<\frac{1}{2}+\frac{n-1}{r}$ and
$2\le r<\infty$.

 To see this, we can compute, using the Sobolev
embedding on $Q_\be$, that for any $r\ge 2$,
\begin{eqnarray*}
  \|\<x\>^{-\al}e^{i t \D^a} f_{1,N} \|_{L^2_t L^r_x} &\simeq&
  \| \<r_\be\>^{-\al} e^{i t \D^a} f_{1,N} \|_{L^2_t \ell^r_\beta L^r(Q_\beta)}\\
&    \les&\|\<r_\be\>^{-\al} e^{i t \D^a} (1-\Delta)^n f_{1,N}
\|_{L^2_t \ell^r_\beta
L^2(Q_\beta)}\\
&    \les&\||x|^{-\al} e^{i t \D^a} (1-\Delta)^n f_{1,N} \|_{L^2_t
\ell^r
L^2_Q}\\
&    \les& \|\La_\omega^{1/2-\al}  (1-\Delta)^n f_{1,N}\|_{ L_x^2}\\
&    \les& \|\La_\omega^{1/2-\al}   f_{1,N}\|_{ L_x^2}\
    .
    \end{eqnarray*}
Then an application of Littlewood-Paley-Stein inequality gives us
\eqref{60-est-Stri-weighted-nonhom} for $r<\infty$.

This gives the proof of \eqref{60-est-Stri-weighted21} in the region
$|x|\ge 1$, for frequency localized functions.

To prove the homogeneous estimates for $|x|\le 1$, we use the
localized weighted HLS estimates proven in Lemma
\ref{lemm-localized-w-HLS}. The estimate is
\beeq\label{60-est-HLS-weighted12} \| |x|^{-\al}
f_{1,N}\|_{L^r_{x\in B_1}}\les \sum_{|\be|\le s}\| |x|^{-\al}\pa^\be
f_{1,N}\|_{L^q_{B_2}},\eneq
 if  $1<q\le r<\infty$, $-n/q'< \al
<n/r$ and $s$ is large enough. Then by
\eqref{60-est-Stri-generalized21} and \eqref{60-est-HLS-weighted12} with $q=2$, we have
\begin{eqnarray*} \||x|^{-\al}e^{i t \D^a} f_{1,N} \|_{L^2_t
L^r_{B_1}} &\les&
    \sum_{|\be|\le s} \| |x|^{-\al} e^{i t \D^a}\pa^\be  f_{1,N} \|_{L^2_t L^2_{B_2}}\\
    &\les& \sum_{|\be|\le s} \||x|^{-\al} e^{i t
\D^a} \pa^\be  f_{1,N} \|_{L^2_t \ell^{r} L^2_Q} \\
&\les&
   \sum_{|\be|\le s}\|\La_\omega^{{\frac 1 2}-\al} \pa^\be f_{1,N} \|_{ L_x^2}
    .\end{eqnarray*}
Recall that $[\Omega_{ij},\pa_k]=\delta_{jk}\pa_i-\delta_{ik}\pa_j$
and $f_1=\phi * f$ for some spectral localized radial bump function $\phi$,
we have the following estimates for the even positive numbers $s$
$$\|\La_\omega^{s} \pa^\be \phi * f \|_{ L_x^2}\les \|\La_\omega^{s} f \|_{ L_x^2}\ .$$
By duality and interpolation, the same estimates are true for a
general real number $s$. Applying this estimate in the previous
inequalities, we find that
\beeq\label{60-est-Stri-weighted-local}\||x|^{-\al}e^{i t \D^a}
f_{1,N} \|_{L^2_t L^r_{B_1}} \les \|\La_\omega^{{\frac 1 2}-\al}
f_{1,N} \|_{ L_x^2}\eneq if $2\le  r<\infty$, $-n/2< \al <n/r$.

Combining \eqref{60-est-Stri-weighted-nonhom} and
\eqref{60-est-Stri-weighted-local}, we get \beeq
\label{2.17}\||x|^{-\al}e^{i t \D^a} f_{1,N} \|_{L^2_t L^r_x} \les
\|\La_\omega^{{\frac 1 2}-\al} f_{1,N} \|_{ L_x^2}\eneq for
$$\frac{2-n}{2}+\frac{n-1}{r}<\al<\frac{n}{r},\
2\le r<\infty\ .$$
An application of
Littlewood-Paley-Stein inequality gives us \eqref{2.17} for $f_1$
since $2\le r<\infty$.

Now we use the Littlewood-Paley decomposition $f=\sum_{j\in\Z}f_j$,
and apply \eqref{2.17} to get
\begin{equation*}
\||x|^{-\al}e^{itD^a}f\|_{L^2_t L^r_x}
\simeq \|\La_\omega^{{\frac 1 2}-\al}f\|_{\dot B_{2,1}^s}.
\end{equation*}
for $s=\frac {n-a}2+\al-\frac {n}{r}$

Next we can use real interpolation as in section 4 to get the case $q=2<r<\infty$ proved. Recall that we have the weighted Strichartz estimates \eqref{9-est-Stri-FW6} for $2\le q=r\le \infty$, and hence \eqref{60-est-Stri-weighted21} with $2\le q=r<\infty$ (see e.g. \eqref{ineq-angu}), we have the estimate \eqref{60-est-Stri-weighted21} for $2\le q \le r<\infty$, by using real interpolation again.

\subsection{Weighted Estimates for $q\ge r$}
In this subsection we will prove Theorem
\ref{8-thm-StriRod-Weig2}.

Recall that Hardy's inequality gives,
\beeq\label{1}\||x|^{-\be}e^{itD^a}f\|_{L^\infty_t
L^2_{|x|}L^2_w}\les \|e^{itD^a}f\|_{L^\infty_t\dot H^\be}\les
\|f\|_{\dot H^\be}\simeq\|2^{j\be}r^{\frac{n-1}2}f_{jk}\|_{\ell_k^2
\ell_j^2 L_r^2}\eneq for $\be\in [0, \frac n 2)$, where we have used
Littlewood-Paley-Stein decomposition for $f$ in the last inequality.
Also we can rewrite the weighted Strichartz estimates
\eqref{9-est-Stri-FW6} as
\begin{eqnarray}\label{2}\||x|^{\frac{n}{2}-\frac{n}{p}-\ga} e^{i t \D^a}
f (x) \|_{L^p_{t, |x| } L^2_\omega } &\les&
\|\D^{\ga-\frac{a}{p}}\La_\omega^{\frac 12-\ga}f\|_{L^2_x}\nonumber\\
&=&\|2^{k(\frac 12-\ga)}2^{j(\ga-\frac a
p)}r^{\frac{n-1}2}f_{jk}\|_{\ell_k^2 \ell_j^2 L_r^2},
\end{eqnarray}
if $\ga\in (\frac 12, \frac n2) $, $a>0$ and $p\in [2,\infty]$.

Now for fixed $q\ge r\ge 2$, if we set $\theta=1-2(\frac 1 r-\frac 1 q), p=q-\frac{2q}r+2$,
by complex interpolation between \eqref{1} and \eqref{2} and using
Theorem 5.6.3 in \cite{BeLo}, we get
\begin{eqnarray*}
\||x|^{-\al} e^{i t \D^a}
f (x) \|_{L^q_{t}L^r_{|x|} L^2_\omega }&\les& \|f_{jk}\|_{\ell^{(\frac 12-\ga)\theta}_{2,k}((\ell^\be_{2,j}L^2_r d\mu,\; \ell^{\ga-\frac a p }_{2,j}L^2_r d\mu)_{[\theta]})}\\
&=& \|f_{jk}\|_{\ell^{b}_{2,k}\ell^{s}_{2,j}L^2_r }\\
&=& \|2^{k b}2^{js}r^{\frac{n-1}2}f_{jk}\|_{\ell_k^2 \ell_j^2L_r^2}\\
&=&\|f\|_{\dot H^{s,b}_\omega},
\end{eqnarray*}
where $d\mu =r^{n-1} dr$, and
$$-\al+\frac a q+\frac n r=-s+\frac n 2\ ,$$
$$2\le r\leq q,\; \frac 1 q-\frac{n-1}2+\frac{n-1}r<\al<\frac n r\ ,$$
 $$\left\{\begin{array}{ll} b \ge -\al+\frac 1q-(n-1)\left(\frac 12-\frac 1r
\right),& \textrm{if } \frac 1 q-\frac{n-1}2+\frac{n-1}r<\al<\frac n q,\\
b>-\frac
{n-1}q-(n-1)\left(\frac 12-\frac
1r\right), &\textrm{if } \frac n q \le \al<\frac n
r.   \end{array}\right.$$

\section{An Application: Strauss Conjecture when $n=2,3$}
As an application of the generalized Strichartz estimates in Theorem \ref{9-thm-Stri-Rod},
Theorem \ref{9-thm-Stri-Rod-local} and Theorem \ref{8-thm-StriSt},
 we prove the Strauss conjecture with low regularity when $n=2,3$.

Let $n= 2, 3$, $s_c=\frac{n}{2}-\frac{2}{p-1}$ be the critical index
of regularity, $s_{d}=\frac{1}{2}-\frac{1}{p}$,
$p_{conf}=1+\frac{4}{n-1}$ be the conformal index, $p_c$ be the
solution of $s_c=s_{d}$ and $p>1$. Let $F_p(u)$ be a function such
that \beeq\label{8-forcingterm}|F_p(u)|\le C|u|^p,\quad |F_p'(u)|\le
C |u|^{p-1}\eneq for some $C>0$ and $p>1$, consider the following
semilinear wave equations \beeq \label{8-eqn-SLW}
\left\{\begin{array}{l} (\pt^2-\Delta ) u = F_p (u)\\
u(0,x)=f, \pt u(0,x)=g.
\end{array}\right. \eneq
Strauss conjecture asserts that for $p>p_c$, the problem
\eqref{8-eqn-SLW} has a global solution, when the initial data
$(f,g)$ is sufficiently small and smooth. This conjecture was
verified by Georgiev, Lindblad and Sogge in \cite{GLS97} and Tataru
\cite{Ta01-2} for smooth data. Then the remaining problem is the
case of low regularity. There has been many partial results on this
field.

When $p\ge p_{conf}$, we have the global result with small $(f,g)\in
\dot{H}^{s_c}\times \dot{H}^{s_c-1}$ in Lindblad-Sogge
\cite{LdSo95}. When the initial data $(f,g)\in
\dot{H}_{rad}^{s_c}\times \dot{H}_{rad}^{s_c-1}$ (here we use the
subscript ``rad" to emphasize that the function is radial), the
results are known when $s_c>\frac{1}{2n}$ or $n\le 4$, see
Lindblad-Sogge \cite{LdSo95}, Sogge \cite{So08} and Hidano
\cite{Hi07}. Then Fang-Wang \cite{FaWa} and
Hidano-Metcalfe-Smith-Sogge-Zhou \cite{HMSSZ} proved the case with $2
\le n \le 4$ for small initial data $(f,g)\in
\dot{H}^{s_c,b}_\omega \times \dot{H}^{s_c-1,b}_\omega$ with
certain $b$, by proving certain $|x|$-weighted Strichartz
estimates  like \eqref{9-est-Stri-FW6}. When $2<p<p_c$ and $n=2,3$,
the radial local results with low regularity $s_{d}$ have been
obtained in Theorem 4.1 of Sogge \cite{So08} ($n=3$) and Hidano-Kurokawa
\cite{Hi08} ($n=2$).

Now we present our results on Strauss conjecture.
\begin{thm}\label{8-thm-Strauss}
Let $n=2,3$,
$$\frac{1}{q}=\left\{\begin{array}{ll}
\frac{2}{p-1}-(n-2) & p_c<p<p_{conf}\ ,\\
\frac{1}{p}+\frac{3-n }{2} & 2<p<p_c\ ,
\end{array}\right.$$
$D^{s,b}=\dot{H}^{s,b}_\omega\times \dot{H}^{s-1,b}_\omega$, and
$S=L^{qp}_t L^p_{|x|} L^\infty_\omega$ be the solution space. If
$p\in (p_c,p_{conf})$, and $(f,g)\in D^{s_c,b}$ with norm bounded
by $\ep\ll 1$ and $b>\frac{1}{qp}+\frac{n-1}{p}$, then there is a unique global
weak solution $u\in S$ to \eqref{8-eqn-SLW}. Also, if $p\in
(2,p_c)$ and $(f,g)\in D^{s_{d},b}$ with norm bounded by $\ep\ll
1$ and $b>\frac 1{qp}+\frac{n-1}p$, we can have a solution $u\in
S_{T_\ep}$ with $T_\ep=c \ep^{\frac{1}{s_c-s_{d}}}$ and $c\ll 1$.
\end{thm}
\begin{rem}
The lifespan given in this result is essentially sharp. In
\cite{Zh92} and \cite{Zh93}, Zhou proved that the life span $T_\ep$
of classical solutions to the equation \eqref{8-eqn-SLW} has order
$\ep^{\frac{1}{s_c-s_{d}}}$ when $n=2,3$ and $2<p<p_c$ (see also
Lindblad \cite{Ld90} for the case $n=3$).
\end{rem}
\begin{rem}
  The choice of $b$ can be possibly improved, if we apply Theorem \ref{9-thm-Stri-Ster} and \ref{8-thm-StriSt}.
\end{rem}
\begin{rem}
  Theorem \ref{8-thm-StriRod-Weig3} can also be employed to give an alternative proof of
  Theorem \ref{8-thm-Strauss} for $2<p<p_c$. See Yu \cite{Yu09} for the related argument for $n=3$.
\end{rem}

For the endpoint case $p=p_c$, we can also get the almost global
result when $n=2$. Let $X^{b}=H^{s_c+\delta,b}_\omega \times
H^{s_c-1+\delta,b}_\omega$ with $\delta >0$. Then we have
\begin{thm}\label{8-thm-Strauss-enpt}
Let $n=2$,  $q=\frac{p-1}{2}$, $p=p_c$, $S=L^{qp}_T L^p_{|x|}
L^\infty_\omega$ and $b>\frac{1}{qp}+\frac{1}{p}$. Suppose that
$(f,g)\in X^{b}$ with norm bounded by $\ep\ll 1$, then there is a
unique almost global weak solution $u\in S_{T_\ep}$ to
\eqref{8-eqn-SLW}, with $T_\ep= exp(c \ep^{-(p-1)^2/2})$ and $c\ll
1$.
\end{thm}

\begin{rem} The sharp lifespan has been proven to be $T_\ep= exp(C \ep^{-p(p-1)})$ for $n\le 3$ (see Zhou \cite{Zh93}). For the higher dimension $n\le 8$, Lindblad and Sogge \cite{LdSo96} proved the existence results with $T_\ep= exp(C \ep^{-p(p-1)})$ (together with the radial results for any dimension $n\ge 3$ with same lifespan).
In \cite{Ta01-2}, Tataru proved the almost global result with
$T_\ep= exp(C \ep^{-r})$ with $r=(p_c^2-1)/(3 p_c+1)$ for small
smooth data and all dimensions.
\end{rem}

In this section, we give a proof of the Strauss conjecture with low
regularity when the dimension is $2, 3$, i.e., Theorem
\ref{8-thm-Strauss}, by using essentially only the generalized
Strichartz estimates. Moreover, we can also prove the almost global
result Theorem \ref{8-thm-Strauss-enpt} for the endpoint case
$p=p_c$ and $n=2$, by using local in time generalized Strichartz estimates.

\subsection{Global Results for $p> p_c$}

When $p\in
(p_c,p_{conf})$, we have $\frac{1}{q}=\frac{2}{p-1}-(n-2)$.

If $u\in S$, we define $\Pi u$ to be the solution of the wave
equation
$$\Box \Pi u=F_p(u)$$ with initial data $(f,g)\in D^{s_c, b}$. Then it suffices to show that when the initial data is small enough in $D^{s_c,b}$, then $\Pi:S\rightarrow S$ and the map is a contraction map on small balls of $S$.

Now we prove this claim. First, note that $\Pi u=v_{h}+v_{i}$ with
$v_h$ being the solution to the homogeneous equation with initial data
$(f,g)\in D^{s_c,b}$, and $v_i$ the solution to the inhomogeneous
equation with null initial data $(0,0)$.

If $n=2$, $u$ is radial and $p_c<p<p_{conf}$, then $(qp,p,s)$ and
$(q',\infty,1-s)$ both satisfy the conditions in Theorem
\ref{8-thm-StriSt}, i.e.,
$$\frac{1}{q p}+\frac{1}{p}<\frac{1}{2}\textrm{ and } q<2.$$
Thus by Christ-Kiselev lemma \cite{ChKi01}, we have for $n=2$
\beeq\label{8-est-inhomo}\|v_i\|_{S} \les  \|F_p(u)\|_{L^{q}_t
L^{1}_{rad}} .\eneq Moreover, if $n=3$, Sogge
(Theorem 4.2 of \cite{So08}) proves the same radial inhomogeneous inequality
\eqref{8-est-inhomo}. Then by the comparison principle for the wave
equation with $n=2,3$, we have
$$\|v_i\|_{S} \les  \|F_p(u)\|_{L^{q}_t L^{1}_{|x|} L^\infty_\omega}
\les \|u\|^p_{S}.$$

Since $(f,g)\in D^{s_c,b}$ with $b>\frac{1}{pq}+\frac {n-1}p$, an
application of Theorem \ref{9-thm-Stri-Rod} and Sobolev embedding
on the sphere yields that
$$\|v_h\|_{S} \les  \|(f,g)\|_{D^{s_c,b}}.$$
Thus we know that \beeq\label{8-est-Strauss-ess} \|\Pi u\|_{S} \les
\|(f,g)\|_{D^{s_c,b}}+ \|u\|^p_{S}.\eneq From this inequality, our
claim follows immediately (recall \eqref{8-forcingterm}).

\subsection{Local Results for $p\in (2,p_c)$}
In this subsection, we prove the local results in Theorem
\ref{8-thm-Strauss} when $p\in (2,p_{c})$. Define
$\frac{1}{q}=\frac{1}{p}+\frac{3-n }{2}$,
$D^{s_d,b}=\dot{H}^{s_{d},b}_\omega\times
\dot{H}^{s_{d}-1,b}_\omega$ and let $S_{T_\ep}=L^{qp}_{T_\ep}
L^p_{|x|} L^\infty_\omega$ be the solution space with
$T_\ep=c\ep^{\frac{1}{s_c-s_{d}}}$, and let $\ep$ be the norm of the
data in $D^{s_{d},b}$. We want
to prove that the map $\Pi$ is a contraction map on small balls of $S_{T_\ep}$.

If $n=2$ and $u$ is radial, let $s=s_d=1/2-1/p$, then $(qp,p,s)$ and
$(q',\infty,1-s)$ satisfy the condition for \eqref{8-est-StriSt-2D2}
and \eqref{8-est-StriSt-2D}. Thus by Christ-Kiselev lemma \cite{ChKi01}, we have for $n=2$
\beeq\label{8-est-inhomo2}\|v_i\|_{S_T} \le C  T^{\frac{1}{q p
}+\frac{n-1}{p}-\frac{n-1}{2}} \|F_p(u)\|_{L^{q}_T
L^{1}_{rad}}.\eneq Moreover, if $n=3$, Sogge (Theorem 4.2 in
\cite{So08}) proves the same radial inhomogeneous inequality
\eqref{8-est-inhomo2}. Then by the comparison principle for wave
equation with $n=2,3$, we have
$$\|v_i\|_{S_T} \le C   T^{\frac{1}{q p}+\frac{n-1}{p}-\frac{n-1}{2}} \|F_p(u)\|_{L^{q}_T L^{1}_{|x|} L^\infty_\omega}
\le C T^{\frac{1}{q p}+\frac{n-1}{p}-\frac{n-1}{2}} \|u\|^p_{S_T}.$$

Since $(f,g)\in D^{s_{d},b}$ with $b>\frac{1}{qp}+\frac{1}{p}$
and norm $\ep$, an application of Theorem \ref{9-thm-Stri-Rod-local}
yields that $$\|v_h\|_{S_T} \le C T^{\frac{1}{q p
}+\frac{n-1}{p}-\frac{n-1}{2}}  \|(f,g)\|_{D^{s_d,b}}\le  C
T^{\frac{1}{qp }+\frac{n-1}{p}-\frac{n-1}{2}} \ep.$$ Thus we know
that \beeq\label{8-est-Strauss-ess2} \|\Pi u\|_{S_T} \le C
T^{\frac{1}{q p}+\frac{n-1}{p}-\frac{n-1}{2}} \ep+  C T^{\frac{1}{qp
}+\frac{n-1}{p}-\frac{n-1}{2}}  \|u\|^p_{S_T}.\eneq Moreover, we
have (recall \eqref{8-forcingterm}) \beeq\label{8-est-Strauss-ess3}
\|\Pi u-\Pi v\|_{S_T} \le  C T^{\frac{1}{qp
}+\frac{n-1}{p}-\frac{n-1}{2}} (\|u\|_{S_T}+\|v\|_{S_T})^{p-1}
\|u-v\|_{S_T}.\eneq

Thus if we choose $T_\ep=c \ep^{\frac{1}{s_c-s_{d}}}$ with $c$
sufficiently small, we can see from \eqref{8-est-Strauss-ess2} and \eqref{8-est-Strauss-ess3} that $\Pi$
is a contraction map on the complete set $$\{u\in S_{T_\ep}|
\|u\|_{S_{T_\ep}}\le 2 C T^{\frac{1}{qp
}+\frac{n-1}{p}-\frac{n-1}{2}} \ep \simeq \ep^{\frac{1}{p}}\},$$
which completes the proof of Theorem \ref{8-thm-Strauss}.

\subsection{Almost Global Results for $p=p_c$ and $n=2$}
In this subsection, we prove the almost global results in Theorem
\ref{8-thm-Strauss-enpt}. Recall that when $p=p_c$ and $n=2$, we
have the local in time estimates \eqref{8-est-StriRod-bdry}.

We use a similar argument to prove this result. Since $(f,g)\in
X^{b}$ with $b>\frac{1}{qp}+\frac{1}{p}$, an application of
Theorem \ref{9-thm-Stri-Rod-local} yields that
$$\|v_h\|_{S_{T_\ep}} \les (\ln(2+T_\ep))^{1/{qp}} \|(f,g)\|_{X^{b}}\les (\ln(2+T_\ep))^{1/{qp}}\ep.$$

Similarly, we have
$$\|v_i\|_{S_{T_\ep}} \les  (\ln(2+T_\ep))^{1/{qp}} \|F_p(u)\|_{L^{q}_t L^{1}_{|x|} L^\infty_\omega}
\les (\ln(2+T_\ep))^{1/{qp}} \|u\|^p_{S_{T_\ep}}.$$ By Theorem \ref{8-thm-StriSt} in the
radial case and the comparison principle.

Now if we choose $T_\ep= exp(c \ep^{-q(p-1)})$ such that
$$\|v_h\|_{S_{T_\ep}} \le \ep^{1/p}.$$ Then if $\|u\|_{S_{T_\ep}}\le 2 \ep^{1/p}$, we have $$\|\Pi u\|_{S_{T_\ep}}\le \ep^{1/p}+  (\ep^{1/p})^p \ep^{-(p-1)/p}\le  2 \ep^{1/p}. $$

Moreover, if $u, v\in S_{T_\ep}$ with norm bounded by $  2 \ep^{1/p}$, then
$$\|\Pi (u-v)\|_{S_{T_\ep}}\le C (\ln(2+T_\ep))^{1/{qp}} (4 \ep^{1/p})^{p-1} \|u-v\|_{S_{T_\ep}}\le \frac{1}{2} \|u-v\|_{S_{T_\ep}},$$
where we have used the assumption \eqref{8-forcingterm}. Thus the
map $\Pi$ is a contraction map on $S_{T_\ep}$ with norm bounded by $  2
\ep^{1/p}$. This completes the proof of  Theorem
\ref{8-thm-Strauss-enpt}.

\section{Appendix}\label{sec-app}
\subsection{Sobolev embedding on the sphere $\Sp^{n-1}$}

We give a simple proof of the Sobolev inequalities used in
of Section \ref{sec-2.1} and \ref{sec-5.2}. These inequalities should be true in
general. For the sake of completeness, we give a proof here.

\begin{lem}[Sobolev embedding I]\label{lem-7.1}
Let $2\le q<\infty$, we have the following \beeq\label{7.1}
\|f\|_{L^q(\Sp^{n-1})}\le C \|\La_\omega^{\sigma(q)}
f\|_{L^2(\Sp^{n-1})}\ ,\eneq where
$$\sigma(q)=(n-1)(\frac 12-\frac 1q)\ .$$
\end{lem}

\begin{prf}
Our proof is based on the spectral cluster estimates on the sphere
(see e.g. \cite{So93} Lemma 4.2.4 page 129). Recall that if we
define the spectral cluster operator
$$\chi_\la f=\sum_{\la_k\in [\la, \la+1]}E_k f$$ where $E_k$ is the
projection onto the one dimensional eigenspace with eigenvalue
$\la_k$, then we have \beeq\label{7.2}\|\chi_\la
f\|_{L^\infty(\Sp^{n-1})}\le C (1+\la)^{(n-2)/2}\|\chi_\la
f\|_{L^2(\Sp^{n-1})}\ .\eneq Thus for the Littlewood-Paley projector
$S_\la$ on the the spectral interval $[\la, 2\la]$, by using
the Cauchy-Schwartz inequality in $j$ with $j\in\mathbb{Z}$, we have
\beeq\label{7.3}\|S_\la f \|_{L^\infty} \le C
(1+\la)^{(n-2)/2}\sum_{j\in [\la, 2\la]} \|\chi_{j} f\|_{L^2} \le
C(1+\la)^{(n-1)/2}\| S_\la f\|_{L^2}\ .\eneq

Then by H\"{o}lder's inequality, we have \beeq\label{7.4}\|S_\la f
\|_{L^q} \le \|S_\la f \|_{L^\infty}^{1-2/q}\|S_\la f
\|_{L^2}^{2/q}\le  C(1+\la)^{(n-1)(1/2-1/q)}\| S_\la f\|_{L^2}\eneq
for any $2\le q\le \infty$.

Finally, by the Littlewood-Paley-Stein theorem (see Theorem 2 \cite{Stri72})
for the sphere, we have \beeq\label{7.4-2}\|f \|_{L^q} \sim \|S_\la
f \|_{L^q \ell^2_\la}\le \| S_\la f\|_{\ell^2_\la L^q}\le C
\|f\|_{H^{\sigma(q)}} \eneq for any $2\le q<\infty$, which
completes the proof.
\end{prf}
\begin{lem}[Sobolev embedding II]\label{lem-7.2}
Let $p\geq 2$,  then for the Littlewood-Paley projector
$S_\la$ on the the spectral interval $[\la, 2\la]$
we have the following
\beeq\label{7.6}
\|S_{\lambda}f\|_{L^{\infty}(\Sp^{n-1})}\le C(1+\lambda)^{(n-1)/p}\|f\|_{L^p(\Sp^{n-1})}
\ .\eneq
\end{lem}

\begin{prf}
Since the estimate for $\la\les 1$ is trivial, we assume that $\la\gg 1$. Note that it is equivalent to prove
the dual estimate of~\eqref{7.6},
$$\|S_{\lambda}f\|_{L^{p'}(\Sp^{n-1})}\le C\lambda^{(n-1)/p}\|f\|_{L^1(\Sp^{n-1})}$$
which is a consequence of  the interpolation between the dual of~\eqref{7.3}, which says
$$\|S_{\lambda}f\|_{L^{2}(\Sp^{n-1})}\le C\lambda^{(n-1)/2}\|f\|_{L^1(\Sp^{n-1})}$$
and
\beeq\label{7.7}
\|S_{\lambda}f\|_{L^{1}(\Sp^{n-1})}\le C\|f\|_{L^1(\Sp^{n-1})}\ .\eneq

Hence we need only to prove~\eqref{7.7}. Let  $P=\sqrt{-\triangle_\omega}$ and $\beta\in C^{\infty}_{0}$ be an even function on $\R$ with support in $\pm(1,2)$. Then
$$S_{\lambda}f=\beta^2 ({P}/{\lambda}) f(x)
=\frac{1}{2\pi}\int_{\mathbb{R}}\lambda\widehat{\beta}(\lambda t) e^{itP} \beta({P}/{\lambda}) f(x)dt.$$
Note that proving~\eqref{7.7} is equivalent to considering
$$T_{\lambda}(P)f(x)=\int_{\mathbb{R}}\lambda\widehat{\beta}(\lambda t)\cos (tP) \beta({P}/{\lambda}) f(x)dt,$$
and proving
\begin{equation}\label{E:estimateofT}
\|T_{\lambda}(P)f\|_{L^{1}(\Sp^{n-1})}\leq C\|f\|_{L^{1}(\Sp^{n-1})}\ .
\end{equation}
Here
$$\cos tP f(x)=\sum_{k=1}^{\infty} \cos t\lambda_kE_k( f)(x)=u(t,x)$$
is the cosine transform of $ f$.
It is the solution of the wave equation
$$(\partial_t^2-\triangle_\omega) u=0\;,\;u(0,\cdot)= f\;,\;u_t(0,\cdot)=0.$$

In order to prove~\eqref{E:estimateofT} , we shall use the finite propagation speed for
solutions to the wave equation. Specifically, if
$ f$ is supported in a geodesic ball $B(x_0,R)$ centered at $x_0$ with radius $R$, then
$x\rightarrow \cos tP f$ vanishes outside of $B(x_0, R+T)$ if $0\leq t\leq T$.\par

Let $1=\eta(t)+\sum_{j=1}^{\infty}\rho(2^{-j}t)$ be a Littlewood-Paley partition of $\mathbb{R}$. Write
$T_{\la}=T_{\la}^0+\sum_{j\ge 1} T_{\la}^j$, where
\begin{equation}\label{E:Tzero}
T_{\lambda}^0(P) f=\int_{\mathbb{R}}\eta(\lambda t)\lambda\widehat{\beta}(\lambda t)\cos(tP) \beta ({P}/{\lambda}) f dt
\end{equation} and
\begin{equation}\label{E:Tj}T_{\lambda}^j(P) f=\int_{\mathbb{R}}\rho(2^{-j}\lambda t)
\lambda\widehat{\beta}(\lambda t)\cos(tP) \beta({P}/{\lambda}) f dt
\end{equation}
We will prove $T_{\lambda}(P)$ satisfies~\eqref{E:estimateofT} by showing $T_{\lambda}^0(P)$ and
$\sum_{j\geq1}T_{\lambda}^j(P)$ both satisfy~\eqref{E:estimateofT}.\par

Now
$$
\begin{aligned}
T_{\lambda}^0(P) f(x)&=\int_{\mathbb{R}}\eta(\lambda t)\lambda\widehat{\beta}(\lambda t)\cos(tP) \beta ({P}/{\lambda}) f(x)dt\\
&=\int_{\mathbb{R}}\eta(\lambda t)\lambda\widehat{\beta}(\lambda t)\sum_{\lambda\leq \lambda_k\leq 2\lambda}
 \cos (t\lambda_k) \beta(\la_k/\la) e_{_k}(x)\int_{\Sp^{n-1}} e_{k}(y) f(y)dydt\\
&=\int_{\Sp^{n-1}}\{\int_{\mathbb{R}}\eta(\lambda t)\lambda\widehat{\beta}(\lambda t)\sum_{\lambda\leq \lambda_k\leq 2\lambda}
 \cos (t\la_k) \beta(\la_k/\la) e_{k}(x) e_{k}(y)dt\} f(y)dy\\
&=\int_{\Sp^{n-1}} K_{\lambda}^0(x,y)f(y)dy
\end{aligned}
$$
The finite propagation speed of the wave equation mentioned before implies that
the kernel $K_{\lambda}^0(x,y)$ must satisfy
$$K_{\lambda}^0(x,y)=0 \;\;\;\;{\rm if}\;\;\;\; {\rm dist}(x,y)>8\lambda^{-1},$$
since $\cos tP$ will have a kernel that vanishes on this set when $t$ belongs to the support of
the integral defining $K_{\lambda}^0(x,y)$. Because of this, in order to prove  $T_{\la}^0$ satisfies~\eqref{E:estimateofT}, it suffices to show that for all geodesic balls $B_{\la,0}$ with radius
$8\la^{-1}$ one has the bound
\begin{equation}\label{E:estimateofTzero}
\|{T_{\la}^0 f}\|_{L^1(B_{\la,0})}\leq C\|f\|_{L^{1}(\Sp^{n-1})},
\end{equation}
for the $L^1$ norm over $B_{\la,0}$.

By using the Cauchy-Schwartz inequality, the dual of \eqref{7.2}, and orthogonality, we can deduce that
\begin{equation}\label{E:estmain}
\begin{split}
\|{T_{\la}^0f}\|_{L^1(B_{\la,0})}&\leq C {\la}^{-(n-1)/2}\|{T_{\la}^0f}\|_{L^2(\Sp^{n-1})}\\
&\leq C {\la}^{-(n-1)/2}\left(\sum_{l=\la}^{2\la}
\|{\chi_{l} f}\|^2_{L^2(\Sp^{n-1})}\right)^{1/2}\\
&\leq C{\la}^{-(n-1)/2}{\la}^{1/2}{\la}^{(n-2)/2}\|{f}\|_{L^1(\Sp^{n-1})}\\
&\leq C \|f\|_{L^1(\Sp^{n-1})}.
\end{split}
\end{equation}

Similarly,
\begin{equation}\label{E:Linfiniteboundthree}
\begin{aligned}
T_{\lambda}^j f(x)&=\int_{\mathbb{R}}\rho(2^{-j}\lambda t)\lambda\widehat{\beta}(\lambda t)\cos(tP)\beta(P/\la) f(x)dt\\
&=\int_{\Sp^{n-1}}\{\int_{\mathbb{R}}\rho(2^{-j}\lambda t)\lambda\widehat{\beta}(\lambda t)
\sum_{\lambda\leq \lambda_k\leq 2\lambda} \cos (t\lambda_k)\be(\la_k/\la) e_{k}(x) e_{k}(y)dt\} f(y)dy\\
&=\int_{\Sp^{n-1}} K_{\lambda}^j(x,y)f(y)dy
\end{aligned}
\end{equation}
has the property that $K_{\lambda}^j(x,y)=0$ if
${\rm dist(x,y)}\geq 8\cdot 2^{j+1}\cdot\lambda^{-1}$.
Note that the dyadic cutoff localizes to $|t|\approx \lambda^{-1}2^j$. Hence
the arguments in \eqref{E:estmain} yields the bound
$({2^{j}}{\la}^{-1})^{(n-1)/2}\la^{1/2}(2^j)^{-N}(\lambda^{(n-2)/2})\|f\|_{L^1}$, if $N$ is
a large enough integer. Here the term $2^{j}\la^{-1}$ comes
from the volume of geodesic ball $B_{\la,j}$ with radius $8\cdot 2^{j+1}\cdot\lambda^{-1}$, and $(2^j)^{-N}$ comes from $\beta \in \mathcal{S}$.
 Thus we have
$$\|T_{\la}^j f\|_{L^{1}}\lesssim 2^{-jN}\|f\|_{L^{1}}.$$
This forms a
geometric series and thus the sum over $j=1,\cdots,\infty$ terms enjoys the
property~\eqref{E:estimateofT}. \par

\end{prf}

\subsection{Morawetz-KSS estimates}\label{app.2}

The Morawetz-KSS estimates are in fact an easy consequence of the energy estimate and the local energy estimate. Recall that the local energy estimate can be stated as follows (see e.g. (1.10) in \cite{FaWa})
\beeq\label{8-est-Morawetz-Local} R^{-\frac{1}{2}} \| e^{i t \D}
f\|_{L^2_{t, x: |x|\le R} } \les \| f\|_{L^2_x} .\eneq
Recall also
that we have the energy estimate
$$T^{-1/2}\|e^{it \D}f\|_{L^2_{t\in [0,T]}
L^2_x} \le \|e^{it \D}f\|_{L^\infty_t
L^2_x}\le \|f\|_{L^2_x}.$$

The Morawetz estimates with $\mu>1/2$ can be proven as follows,
$$\|\<x\>^{-\mu}e^{it\D}f\|_{L^2_{t,x}}\les \sum_{j\ge 0}2^{- j\mu}\| e^{i t \D}
f\|_{L^2_{t, x: |x|\le 2^j} }\les  \sum_{j\ge 0}2^{- j(\mu-1/2)}\|f\|_{L^2_x}\les \|f\|_{L^2_x}\ .$$

Now we deal with the case $\mu\le 1/2$.
We consider first the case when $T\les 1$. In this case, the
Morawetz-KSS estimates are in fact weaker than the energy estimate,
$$\| \langle
x\rangle^{-\mu} e^{i t \D} f\|_{L_{[0,T]}^2 L^2_{x}}\les
T^{\frac{1}{2}} \|f\|_{L^2_x} \les  A_\mu (T) \|f\|_{L^2_x}.$$

For the remaining case with $T\ge 2$, we use the energy estimate to deal with the region $|x|\ge T$, $$\| \langle x\rangle^{-\mu} e^{i t \D}
f\|_{L_{[0,T]}^2 L^2_{x: |x|\ge T}}\les T^{-\mu} \|e^{i t \D}
f\|_{L_{[0,T]}^2 L^2_{x}}\les T^{\frac{1}{2}-\mu}\|f\|_{L^2_x}\les
A_\mu(T) \|f\|_{L^2_x}.$$ For the remaining region, we use instead
\eqref{8-est-Morawetz-Local}
\begin{eqnarray*}
    \| \langle
x\rangle^{-\mu} e^{i t \D} f\|_{L_{[0,T]}^2 L^2_{x}}^2 &\les& \sum_{
0\le j \les \ln T } 2^{-2 j \mu} \| e^{i t \D} f\|^2_{L_{T}^2
L^2_{|x|\le 2^j}}\\
& \les&  \sum_{ 0\le j \les \ln T } 2^{j(1- 2 \mu)}\|
 f\|_{L^2_x}^2 \\
 &\les&
A_\mu(T)^2 \|f\|^2_{L^2_x}.
\end{eqnarray*}
This completes the proof of \eqref{8-est-KSS-gene}.

\section*{Acknowledgments}
 The authors would like to
thank Carlos Kenig, Chris Sogge and Terry Tao for helpful conversations in connection with this work. They are also grateful to Sinan Ariturk's help in preparing this paper.

\end{document}